\documentclass[10pt]{article}
\usepackage{amsmath, a4, latexsym,amssymb}
\usepackage{color, enumerate}
\def\qed{\mbox{$\Box$}}
\def\boxp{\mbox{$\Box$}}
\newenvironment{Proof}{\par\noindent{\bf Proof.\ }}{}
\newenvironment{proof}{\par\noindent{\bf Proof.\ }}{\hfill\qed\\ }

\begin{document}
\title{ Infinite rate mutually catalytic branching
in infinitely many colonies. Construction, characterization and
convergence
}
\author{
Achim Klenke
\\Institut f{\"u}r Mathematik
\\Johannes Gutenberg-Universit{\"a}t Mainz
\\Staudingerweg 9
\\D-55099 Mainz
\\Germany
\\math@aklenke.de
\and
Leonid Mytnik
\\Faculty of Industrial Engineering\\ and Management
\\Technion -- Israel Institute of Technology
\\Haifa 32000\\ Israel
\\leonid@ie.technion.ac.il}
\date{May 30, 2011\\{\small First submitted on January 05, 2009}
}
\maketitle
\mbox{}\vspace{1cm}
\begin{abstract}
We construct a mutually
catalytic branching process on a countable site space with infinite ``branching rate''.
The finite rate mutually catalytic model, in which the rate of branching
of one population at a site is proportional to the mass of the other population at that
site, was  introduced by Dawson and Perkins in~\cite{bib:dp98}. We show that our model is the limit
for a class of models and in particular
for  the Dawson-Perkins model as the rate of branching goes to infinity. Our process is characterized as the unique solution to a martingale problem. We also give a characterization of the process as
a weak solution of an infinite system of stochastic integral equations driven by a Poisson noise.
\end{abstract}
\footnote{This work is partly funded by the German Israeli Foundation with grant
number G-807-227.6/2003}

\thispagestyle{empty}

\clearpage

\newfam\msbmfam\font\tenmsbm=msbm10\textfont
\msbmfam=\tenmsbm\font\sevenmsbm=msbm7
\renewcommand{\theequation}{\arabic{section}.\arabic{equation}}
\renewcommand{\thesection}{\arabic{section}}
\scriptfont\msbmfam=\sevenmsbm\def\bb#1{{\fam\msbmfam #1}}

\newcounter{archapter}[section]
\newtheorem{theorem}{Theorem}[section]
\newtheorem{lemma}[theorem]{Lemma}
\newtheorem{definition}[theorem]{Definition}
\newtheorem{remark}[theorem]{Remark}
\newtheorem{proposition}[theorem]{Proposition}
\newtheorem{case}{Case}[archapter]
\newtheorem{corollary}[theorem]{Corollary}
\newtheorem{fact}[theorem]{Fact}
\newtheorem{assumption}[theorem]{Assumption}
\newcommand{\gdm}{\hfill\qed}
\def\eope{\tag*{\boxp}}

\newcommand{\bS}{\mathbf{S}}
\newcommand{\bA}{\mathbf{A}}
\def\AAb{\mathbb{A}}\def\BB{\mathbb{B}}\def\CC{\mathbb{C}}
\def\EE{\mathbb{E}}\def\GG{\mathbb{G}}\def\HH{\mathbb{H}}\def\KK{\mathbb{K}}
\def\LL{\mathbb{L}}\def\MM{\mathbb{M}}\def\FF{\mathbb{F}}\def\NN{\mathbb{N}}\def\PP{\mathbb{P}}\def\QQ{\mathbb{Q}}
\def\di{D}
\def\dim{\di_m}
\def\LIBE{\LL_\infty^{\beta,E}}
\def\LIBEPP{\LL_\infty^{\beta,E,++}}
\def\LB{\LL^{\beta}}
\def\LBE{\LL^{\beta,E}}
\def\LBT{\LL^{\beta,2}}
\def\LFE{\LL^{f,E}}
\def\LFT{\LL^{f,2}}
\def\LIB{\LL_\infty^{\beta}}
\def\Re{\mathop{\mathrm{Re}}}
\def\Im{\mathrm{Im}}
\def\RR{\mathbb{R}}\def\WW{\mathbb{W}}\def\ZZ{\mathbb{Z}}
\def\VV{\mathbb{V}}
\def\cA{{\mathcal{A}}}\def\cB{{\mathcal{B}}}\def\cL{{\mathcal{L}}}
\def\cD{{\mathcal{D}}}\def\cE{{\mathcal{E}}}\def\cF{{\mathcal{F}}}
\def\cH{{\mathcal{H}}}\def\cI{{\mathcal{I}}}\def\cJ{{\mathcal{J}}}
\def\cL{{\mathcal{L}}}
\def\cM{{{{\color{red}\underbrace{\color{black}\mathcal{M}}_{}}}}}
\def\cN{{{{\color{red}\underbrace{\color{black}\mathcal{N}}_{}}}}}
\def\CF{{\mathcal{F}}}
\def\CN{{\mathcal{N}}}
\def\CM{{\mathcal{M}}}
\def\cP{{\mathcal{P}}}\def\cS{{\mathcal{S}}}\def\cT{{\mathcal{T}}}\def\cV{{\mathcal{V}}}
\def\cW{{\mathcal{W}}}\def\cX{{\mathcal{X}}}\def\cY{{\mathcal{Y}}}\def\cZ{{\mathcal{Z}}}
\def\cG{{\mathcal{G}}}
\def\cA{{\mathcal{A}}}
\def\ep{\varepsilon}
\def\ve{\varepsilon}
\def\ve{\epsilon}
\def\LSC{\Gamma}
\newcommand\ap[1]{\bar{#1}}
\newcommand\apm[1]{#1^{(m)}}
\newcommand\app[1]{#1^{(m,\ve)}}
\newcommand\apv[2]{#1^{(#2)}}
\def\apX{\ap{X}}
\def\apmX{\apv{X}{m}}
\def\apmX{\apv{X}{m}}
\def\apmCA{\apv{\CA}{m}}
\def\apmA{\apv{\bA}{m}}
\def\apmCS{\apv{\cS}{m}}
\def\apmbS{\apv{\bS}{m}}
\def\apmp{\apv{p}{m}}
\def\apmeX{\apv{X}{m,\ve}}
\def\appX{\app{X}}
\def\appIN{I^{0,(m,\ve)}}
\def\apmIN{I^{0,(m)}}
\newcommand{\attop}[2]{\genfrac{}{}{0pt}{1}{#1}{#2}}
\renewcommand{\cases}[1]{\left\{\begin{array}{rl}#1\end{array}\right.}
\newcommand{\equ}[1]{(\ref{#1})}
\newcommand{\mbu}{\quad\mbox{ and }\quad}
\newcommand{\mbs}[1]{\mbox{ \;#1\; }}
\newcommand{\mbsl}[1]{\mbox{ \;#1}}
\newcommand{\mbsr}[1]{\mbox{#1\; }}
\newcommand{\mas}{\quad\mbox{\;as \;}}
\newcommand{\mf}{\quad\mbox{\;for \;}}
\newcommand{\mfa}{\quad\mbox{\;for all \;}}
\newcommand{\mfj}{\quad\mbox{\;for any \;}}
\newcommand{\mfso}{\quad\mbox{\;for some \;}}
\newcommand{\mfev}{\quad\mbox{\;for every \;}}
\newcommand{\mfasts}{\quad\mbox{almost surely}}
\newcommand{\mfastu}{\quad\mbox{almost everywhere}}
\newcommand{\mfu}{\quad\mbox{a.e.}}
\newcommand{\mfs}{\quad\mbox{a.s.}}
\newcommand{\N}{\mathbb{N}}
\newcommand{\R}{\mathbb{R}}
\newcommand{\C}{\mathbb{C}}
\newcommand{\Q}{\mathbb{Q}}
\newcommand{\Z}{\mathbb{Z}}
\newcommand{\limv}{\mathop{\mathrm{v\hspace{-0.04em}{\scriptscriptstyle\raisebox{0.05em}{-}}\hspace{-0.04em}lim}}\limits}
\newcommand{\Dlim}{\,\stackrel{\mathcal{D}}{\longrightarrow}\,}
\newcommand{\Dlimn}
 {\;\mbox{\raisebox{-5pt}{${\attop{\DS\Dlim}{n\to\infty}}$}\;}}
\newcommand{\To}[1]{\,\stackrel{#1}{\longrightarrow}\,}
\newcommand{\Toi}[1]{\To{#1 \rightarrow \infty}}
\newcommand{\Ton}[1]{\To{#1 \rightarrow 0}}
\newcommand{\limn}{\Toi{n}}
\newcommand{\liml}{\Toi{l}}
\newcommand{\limk}{\Toi{k}}
\newcommand{\limt}{\Toi{t}}
\newcommand{\limT}{\Toi{T}}
\newcommand{\limm}{\Toi{m}}
\newcommand{\limN}{\Toi{N}}
\renewcommand{\P}{\mathbf{P}}
\newcommand{\E}{\mathbf{E}}
\newcommand{\Var}{\mathbf{Var}}
\newcommand{\Cov}{\mathbf{Cov}}
\newcommand{\wlimn}{\stackrel{n \rightarrow \infty}{\Longrightarrow}}
\newcommand{\wlimm}{\stackrel{m \rightarrow \infty}{\Longrightarrow}}
\newcommand{\wlimk}{\stackrel{k \rightarrow \infty}{\Longrightarrow}}
\newcommand{\wlimt}{\stackrel{t \rightarrow \infty}{\Longrightarrow}}
\newcommand{\wlimN}{\stackrel{N \rightarrow \infty}{\Longrightarrow}}
\def\CA{\mathcal{A}}
\def\1{\mathbf{1}}
\newcommand{\DS}{\displaystyle}
\newcommand{\TS}{\textstyle}
\newcommand{\ARG}{\,\boldsymbol{\cdot}\,}
\newcommand{\mfalls}{\quad\mbox{if \;}}
\newcommand{\msonst}{\quad\mbox{otherwise}}
\def\mtimes{\diamond}
\def\lb{\langle\hspace*{-0.13em}\langle}
\def\Lb{\big\langle\hspace*{-0.20em}\big\langle}
\def\rb{\rangle\hspace*{-0.13em}\rangle}
\def\Rb{\big\rangle\hspace*{-0.20em}\big\rangle}
\def\lbb{\langle\hspace*{-0.13em}\langle\hspace*{-0.13em}\langle}
\def\Lbb{\big\langle\hspace*{-0.20em}\big\langle\hspace*{-0.20em}\big\langle}
\def\rbb{\rangle\hspace*{-0.13em}\rangle\hspace*{-0.13em}\rangle}
\def\Rbb{\big\rangle\hspace*{-0.20em}\big\rangle\hspace*{-0.20em}\big\rangle}
\section{Introduction and main results}
\label{S1}
\subsection{Background and Motivation}
\label{S1.1}
In~\cite{bib:dp98} Dawson and Perkins considered the following mutually catalytic model:
\begin{equation}
\label{E1.1}
\begin{aligned}
Y_{i,t}(k)\;&=\;Y_{i,0}(k)+\int_0^t \sum_{l\in S} \CA(k,l) Y_{i,s}(l)\,ds \\
&\phantom{=}\;+\int_0^t \big(\gamma Y_{1,s}(k)Y_{2,s}(k)\big)^{1/2}\,dW_{i,s}(k),\;\;t\geq 0, k\in S,\,i=1,2.
\end{aligned}
\end{equation}
Here $S$ is a countable set that is thought of as the site space. (In fact, Dawson and Perkins made the explicit choice $S=\ZZ^d$.) The matrix $\CA$ is defined by
$$\CA(k,l)=a(k,l)-\1_{\{k=l\}},$$
where $a$ is a symmetric transition matrix of a Markov chain on $S$. Finally,  $(W_i(k),\; k\in S,\,i=1,2)$ is an independent family of one-dimensional Brownian motions. Dawson and Perkins studied the long-time behavior of this
model and also constructed the analogous model in the continuous setting on $\R$ instead of $S$.
One can think of $\gamma$ as being the branching rate for this model.

In this paper we study (under weaker assumptions on the matrix $\CA$) a model that formally corresponds to the case $\gamma=\infty$. This infinite rate mutually catalytic branching process can be characterized by a certain martingale problem. We show that this martingale problem is well-posed and its solution $X$ is the unique solution
of a system of stochastic differential equations driven by a certain Poisson noise. In fact, we construct the solution via approximate solutions of this system of SDEs. Furthermore, we show that $X$ is the limit of the Dawson-Perkins processes as $\gamma\to\infty$. Hence, we call $X$ the \emph{infinite rate mutually catalytic branching process} (IMUB).

This is the second part in a series of three papers. In the first part \cite{KM1}, we studied the infinite rate mutually catalytic branching process in the case where $S$ is a singleton. In the third part \cite{KM3}, we investigated the longtime behaviour for the case where $S$ is countable. There we establish a dichotomy between segregation and coexistence of types depending on the potential properties of the migration mechanism $\CA$.

An alternative construction of the infinite rate mutually catalytic branching process via a Trotter type approximation scheme can be found \cite{Oeler2008} and \cite{KO}. We remark that although the approach in \cite{KO} is more easily accessible, it yields less information about the IMUB process than the approach made here. In particular, the investigation of the longtime behaviour in \cite{KM3} needs the description of the jumps of the process that we develop in this paper.

\subsection{Notation}
\label{S1.2}
We have to introduce some notation.
Let $\CA=(\CA(k,l))_{k,l\in S}$ be a matrix on $S$ satisfying the following assumptions:
\begin{equation}
\label{E1.2}
\CA(k,l)\geq 0\mf k\not= l
\end{equation}
and
\begin{equation}
\label{E1.3}
\|\CA\|:=\sup_{k \in S}\sum_{l\in S} |\CA(k,l)| + |\CA(l,k)| <\infty.
\end{equation}
Let
\begin{equation}
\label{E1.4}
E=[0,\infty)^2\setminus (0,\infty)^2.
\end{equation}

For $u,v\in [0,\infty)^S$  define
$$\langle u,v\rangle=\sum_{k\in S}u(k)v(k)\,\in[0,\infty].$$
Similarly, for $x\in([0,\infty)^2)^S$ and $\zeta\in[0,\infty)^S$ define
$$\langle x,\zeta\rangle =\sum_{k\in S}x(k)\zeta(k)\,\in[0,\infty]^2.
$$

By Lemma~IX.1.6 of \cite{Liggett1985}, there exists a $\beta\in(0,\infty)^S$ and an $\LSC\geq1$ such that
\begin{equation}
\label{E1.5}
\sum_{k\in S}\beta(k)< \infty
\end{equation}
and
\begin{equation}
\label{E1.6}
\sum_{l\in S}\beta(l)(|\CA(k,l)|+|\CA(l,k)|)\leq \LSC\beta(k)\mfa k\in S.
\end{equation}

We fix this $\beta$ for the rest of this paper.
Note that for the transpose matrix $\CA^*$ of $\CA$, we have $\|\CA^*\|=\|\CA\|<\infty$ and \equ{E1.6} holds with the same $\beta$. Hence, in what follows, $\CA$ could be replaced by $\CA^*$. We will make use of this fact in Section~\ref{S4} when we construct a dual process.

Let us define the Liggett-Spitzer spaces as follows:
$$\begin{aligned}
\LB&=\big\{ u\in [0,\infty)^S:\;\langle u,\beta\rangle<\infty\big\},\\
\LBT&=\big\{ x\in \big([0,\infty)^2\big)^S:\;\langle x,\beta\rangle\in[0,\infty)^2\big\},\\
\LBE&=\LBT\cap E^S.
\end{aligned}
$$
For $u\in\R^S$, let
\begin{equation}
\label{E1.7}
\left\Vert u\right\Vert_{\beta} =\sum_{k\in S} |u(k)|\beta(k).
\end{equation}
Furthermore, for $x=(x_1,x_2)\in\LBT$, let $\Vert x\Vert_{\beta,2}=\Vert x_1\Vert_\beta+\Vert x_2\Vert_\beta$.
Note that $\Vert\ARG\Vert_\beta$ defines a topology on $\LB$. Furthermore,
$\Vert\ARG\Vert_{\beta,2}$ defines a topology on $\LBT$ and on $\LBE$. We will henceforth assume that these spaces are equipped with these topologies.

Let $\CA f(k)=\sum_{l\in S}\CA(k,l)f(l)$ if the sum is well defined. Let $\CA^n$ denote the $n$th matrix power of $\CA$ (note that this is well defined and finite by \equ{E1.3}) and define
$$p_t(k,l):=e^{t\CA}(k,l):=\sum_{n=0}^\infty \frac{t^n\,\CA^n(k,l)}{n!}.$$
Let $\cS$ denote the (not necessarily Markov) semigroup generated by $\CA$, that is,
$$
\cS_t f(k)= \sum_{l\in S} p_t(k,l)f(l)\mf t\geq 0.
$$
We will use the notation $\CA f$, $\cS_tf$ and so on also for $[0,\infty)^2$ valued functions $f$ with the obvious meaning.

Note that for $f\in \LL^\beta$, the expressions $\CA f$ and $\cS_t f$ are well defined and that (recall $\LSC$ from \equ{E1.6})
\begin{equation}
\label{E1.8}
\Vert\CA f\Vert_\beta \leq \LSC\Vert f\Vert_\beta
\quad\mbox{and}\quad
\Vert\cS_t f\Vert_\beta \leq e^{\LSC t}\Vert f\Vert_\beta.
\end{equation}

Let $\bA(k,l)=\CA(k,l)^+$.
Denote by $(\bS_t)_{t\geq 0}$ the semigroup generated by $\bA$, that is, $\bS_t=\sum_{n=0}^\infty e^{-t}\bA^n/n!$.
Clearly, for any $f\in \LB$ and $k\in S$, we have
$$
\CA f(k)\leq \bA f(k),\qquad
\cS_t f(k)\leq \bS_t f(k)
\mbu
f(k)\leq \bS_t f(k).
$$
As above, it is easy to check that
\begin{equation}
\label{E1.9}
\left\Vert \bA f\right\Vert_{\beta}\leq \LSC\left\Vert f\right\Vert_{\beta}\mbu
\left\Vert \bS_tf\right\Vert_{\beta}\leq e^{\LSC t}\left\Vert f\right\Vert_{\beta}\mfa t\geq 0.
\end{equation}
Therefore, we trivially have
\begin{align}
\label{E1.10}
\bA f(k)&\leq \frac{\LSC\left\Vert f\right\Vert_{\beta}}{\beta(k)}\mfa f\in \LB,\\
\label{E1.11}
\bS_tf(k)&\leq \frac{e^{\LSC t}\left\Vert f\right\Vert_{\beta}}{\beta(k)}\mfa  f\in \LB,
 \;t\geq 0.
\end{align}
All the estimates \equ{E1.8}--\equ{E1.11} also hold for the transposed matrix $\CA^*$ and the derived objects $\bA^*$, $\bS^*$ and so on.

Let $D_{\LBE}=D_{\LBE}[0,\infty)$ be the Skorohod space of c{\`a}dl{\`a}g $\LBE$-valued functions.

We will employ a martingale problem in order to characterize the (bivariate) process $X\in D_{\LBE}$ that will be
the limit of the Dawson-Perkins models as $\gamma\to\infty$.
In order to formulate this martingale problem for $X$, we need some more notation. For $x=(x_1,x_2)$ and $y=(y_1,y_2)\in\R^2$ we introduce the
\emph{lozenge product}
$$
x \mtimes{} y\;:=\;-(x_1+x_2)(y_1+y_2)\;+\;i(x_1-x_2)(y_1-y_2)
$$
(with $i=\sqrt{-1}$) and define
$$
F(x,\,y)=\exp(x\mtimes y).
$$
Note that the lozenge product defines a symmetric bilinear form, in particular $x\mtimes y=y\mtimes x$. Some more properties of $F$ and the lozenge product can be found in \cite[Lemma 2.2, Corollaries 2.3, 2.4]{KM1}.
For $x,y\in(\R^2)^S$, we write
$$
\lb x,\,y\rb\,\;=\;\sum_{k\in S}x(k)\mtimes y(k)
$$
whenever the infinite sum is well defined
and let
\begin{equation}
\label{E1.12}
H(x,\,y)=\exp(\lb x,\,y\rb).
\end{equation}

Define
\begin{equation}
\label{E1.13}
\LFT=\big\{ y\in ([0,\infty)^2)^S:\,y (k)\neq0\mbs{for only finitely many}k\in S\big\}
\end{equation}
and
\begin{equation}
\label{E1.14}
\LFE=\LFT\cap E^S.
\end{equation}
Finally, define the spaces
\begin{equation}
\label{E1.15}
\begin{aligned}
\LIB&=\left\{f\in [0,\infty)^S:\;\langle f,g\rangle<\infty\mfa g\in \LB\right\}\\
&=\Big\{f\in\LB:\sup_{k\in S}f(k)/\beta(k)<\infty\Big\}
\end{aligned}
\end{equation}
and
$$
\LIBE=\left\{ \eta=(\eta_1,\eta_2)\in E^S: \eta_1, \eta_2\in \LIB\right\}.
$$
As a subspace, $\LIB$ inherits the norm of $\LB$.

Note that the function $H(x,y)$ is well defined if either $x\in (\R^2)^S$ and $y\in\LFE$ or $x\in\LBE$ and $y\in\LIBE$.
\subsection{Main Results}
\label{S1.3}
\subsubsection*{Martingale Problem}
Our main theorem is the following.
\begin{theorem}
\label{T1.1}
\textbf{(a)}
For all $x\in \LBE$, there exists a unique solution $X\in D_{\LBE}$ of the following martingale problem:
For each $y\in\LFE$, the process $M^{x,y}$ defined by
\[
\label{MP1}
M^{x,y}_t:=H(X_t,y) -H(x,y)-\int_0^t \lb\CA X_s,y\rb\, H(X_s,y)\,ds\tag{MP$_1$}
\]
is a martingale with $M^{x,y}_0=0$.\smallskip

\textbf{(b)} For any $x\in\LBE$ and $y\in\LIBE$, the process $M^{x,y}$ is well defined and is a martingale.\smallskip

\textbf{(c)} Denote by $P_x$ the distribution of $X$ with $X_0=x$. Then $(P_x)_{x\in\LBE}$ is a strong Markov family.
\end{theorem}
\subsubsection*{Stochastic integral equation}
Unfortunately, the characterization of $X$ as the solution of the martingale
problem \equ{MP1} does not shed much light on properties of the process $X$ such as: Is $X$ continuous or discontinuous? If it is discontinuous, what is the structure of jump formation?

These questions will be answered by a different representation of $X$ as
as a solution to a system of stochastic differential equations of jump type.
We will see that the coordinate processes of $X$ are so-called purely discontinuous martingales and we will give a precise quantitative statement about the distribution of jumps.

Before we give the exact description, let us briefly and roughly recall the concept of stochastic integrals with respect to Poisson point measures. Let $\nu$ be a finite measure on some Borel space $F$ (which will be taken to be $[0,\infty)\times E$ later) and assume that $N$ is a Poisson point process on $[0,\infty)\times F$ with intensity measure $N':=\lambda\otimes\nu$ (here $\lambda$ denotes the Lebesgue measure). Furthermore, let $(Z_t)_{t\geq0}$ be an $\R^F$-valued predictable process (with respect to the filtration $\sigma(N([0,t]\times\ARG))_{t\geq0}$). Then define the integral
$$(Z\ast N)_t(\omega):=\int_0^t\int_F Z_s(\omega;x)N\big(\omega;ds\otimes dx\big)
=\sum_{s\leq t, x\in F}Z_s(\omega;x)N\big(\omega;\{s\}\times\{x\}\big).$$
Note that the sum is finite since the intensity measure $\nu$ is finite. Now, define the so-called martingale measure $M:=N-N'$. Then
$$(Z\ast M)_t:=(Z\ast N)_t-\int_0^t\int_FZ_s(\omega;x)\,N'(ds\otimes dx)$$
is well defined for almost all $\omega$ if, for example,
\begin{equation}
\label{E1.16}
\E\left[\int_0^t\int_F|Z_s(x)|\,N'(ds\otimes dx)\right]<\infty.
\end{equation}
In this case, $Z\ast M$ is an integrable process and is, in fact, a martingale (here we use that $Z$ is predictable). Now, by some $L^1$-approximation procedure, the assumption that $\nu$ be finite could be weakened to $\sigma$-finiteness if condition \equ{E1.16} is fulfilled. In this case, both
$Z\ast N$ and $Z\ast N'$ are well defined. However, using an $L^2$-approximation scheme (similarly as for the construction of infinitely divisible random variables with general L{\'e}vy measure), we can define $Z\ast M$ even if only
\begin{equation}
\label{E1.17}
\E\left[\left(\int_0^t\int_FZ_s(x)^2\,N(ds\otimes dx)\right)^{1/2}\right]<\infty\mfa t\geq0.
\end{equation}
In this case, $Z\ast M$ is still a local $L^2$-martingale. It is purely discontinuous in the sense that it is orthogonal to all continuous local martingales. This general construction of stochastic integrals with respect to integer valued martingale measures is performed in full generality, for example, in \cite[Section II, 1d]{bib:jacshir87}.

The process that we will construct does not have second moments but we will show that it has $p$th moments of all orders $p\in[1,2)$. Hence, checking \equ{E1.17} is a bit tricky. Now,
for $p<2$, we have the simple estimate $(\sum a_i^2)^{1/2}\leq (\sum |a_i|^p)^{1/p}$ for $a_i\in\R$. Hence, using Jensen's inequality, as pointed out in the proof of \cite[Lemma 3.1]{bib:LeGallMytnik},
it is enough to show
that for some $p\in(1,2)$, we have
\begin{equation}
\label{E1.18}
\E\left[\int_0^t\int_F|Z_s(x)|^p\,N'(ds\otimes dx)\right]<\infty\mfa t\geq0.
\end{equation}
In fact, following the proof of \cite[Lemma 3.1]{bib:LeGallMytnik} (see also \cite[Proposition I.1.47(c)]{bib:jacshir87}), one readily gets that if condition \equ{E1.18} holds for all $p\in(1,2)$, then $Z\ast M$ is an $L^p$-martingale for any $p\in[1,2)$.

Now, our aim is to define the process $X$ such that the coordinate processes solve a system of stochastic integral equations where $F=[0,\infty)\times E$.

The first step is, of course, to describe the intensity measure on $F$. Then
we formulate the stochastic integral equation and state in a theorem that it has a unique (weak) solution. The construction of the solution will be performed by an approximation scheme with finite intensity measures and finite site spaces. Uniqueness will be shown using a self-duality of the solution.

The stochastic parts of the single coordinates in the Dawson-Perkins process
defined in \equ{E1.1} are two-dimensional isotropic diffusions and are hence
time-transformed planar Brownian motions. When we speed up these motions, at any positive time, they
will be close to their absorbing points at $E$. Hence, a crucial role in the
subsequent considerations will be played by the harmonic measure $Q$ of planar Brownian
motion $B$ on $(0,\infty)^2$. That is, if $B=(B_1,B_2)$ is a Brownian
motion in $\R^2$ started at $x\in[0,\infty)^2$ and
$\tau=\inf\{t>0: \,B_t\not\in(0,\infty)^2\}$, then we define
$$
Q_x=\P_x[B_\tau \in \ARG].
$$
\begin{lemma}\label{L1.2}
If $x=(u,v)\in(0,\infty)^2$, then the harmonic measure $Q_x$ has a one-dimensional Lebesgue density on $E$
that is given by\begin{equation}
\label{E1.19}
Q_{(u,v)}\big(d(\bar u,\bar v)\big)=
\cases{\DS
\frac4\pi
\,\frac{\TS uv\,\bar u}{\textstyle 4u^2v^2+\big(\bar u^2+ v^2-u^2\big)^2}\;d\bar u,
&\quad\mbox{if }\bar v=0,\\[6mm]
\DS \frac4\pi
\,\frac{\TS uv\,\bar v}
{\TS 4u^2v^2+\big(\bar v^2+u^2-v^2\big)^2}\;d\bar v,
&\quad\mbox{if }\bar u=0.
}
\end{equation}
Furthermore, trivially we have
$Q_x=\delta_x$ if $x\in E$.
\end{lemma}

Formula \equ{E1.19} appears in the remark on page 1094 of~\cite{bib:dp98} and could be derived by recalling that the Cauchy distribution is the harmonic measure for planar Brownian motion on the upper half plane and then applying the conformal map $z\mapsto \sqrt{z}$ (identifying $\R^2$ with $\C$) that maps the half plane to the quadrant.
A more formal proof of this lemma is deferred to the appendix.

As the next goal is to define a measure for the jumps that drive the process $X$, we need to describe the infinitesimal dynamics of $X$. These will be defined in terms of the $\sigma$-finite measure $\nu$
on $E$ that arises as the vague limit (on $E\setminus\{(1,0)\}$) of $\ve^{-1}Q_{(1,\ve)}$ as $\ve\to0$. Using \equ{E1.19}, it is easy to see that $\nu$ has a one-dimensional Lebesgue density given by
\begin{equation}
\label{E1.20}
\nu\big(d(u,v)\big)=\cases{\DS
\frac4\pi\,\frac{u}{(1-u)^2\,(1+u)^2}\;du,&\quad\mbox{if
}v=0,\\[4mm]\DS
\frac4\pi\,\frac{v}{\big(1+v^2\big)^2}\;dv,&\quad\mbox{if }u=0.
}
\end{equation}

We use $\nu$ to define the Poisson point process (PPP)  that will be the driving force of the equations.
Let $\CN$ be the PPP on $S\times \R_+\times\R_+\times  E$ with intensity
\begin{equation}
\label{E1.21}
\CN' = \ell_S \otimes \lambda \otimes \lambda \otimes \nu,
\end{equation}
where $\lambda$ is the Lebesgue measure on $\R_+$ and $\ell_S$ is the counting measure on $S$. The first $\R_+$ is used as time set while the second $\R_+$ is used to model the (predictable) intensity $I(X_{t-};k)$ at which jumps at site $k\in S$ come depending on the current state $X_{t-}$. Now assume that $\FF=(\cF_t)_{t\geq0}$ is a filtration that fulfills the usual hypotheses and is such that
$$
(\CN-\CN')\big(\{k\}\times [0,t]\times A\times B\big)_{t\geq0}
$$
is an $\FF$-martingale for all $k\in S$ and measurable $A\subset\R$ and $B\subset E$ with $\lambda(A)\nu(B)<\infty$.
Finally, define the $\FF$-martingale measure
\begin{equation}
\label{E1.22}
\CM:=\CN-\CN'.
\end{equation}

The measure $\nu$ is the limit of the $Q$ only at the point $(1,0)\in E$. The limits $\nu_{(u,0)}$ of $\ve^{-1}Q_{(u,\ve)}$ and $\nu_{(0,v)}$ of $\ve^{-1}Q_{(\ve,v)}$ can be obtained by simple transformations of $\nu$ (see \cite[discussion before (5.5)]{KM1}): For suitable $f:E\to\R$, we have
$$\int_E f(y)\,\nu_{(u,0)}(dy)\;=\;\frac{1}{u}\int_E f(u(y_1,y_2))\,\nu(dy)$$
and
$$\int_E f(y)\,\nu_{(0,v)}(dy)
\; =\;\frac{1}{v}\int_E f\big(v(y_2,y_1)\big)\,\nu(dy).$$

Hence, if we define the functions
\begin{equation}
\label{E1.23}
J_i(y,z)=y_2z_{3-i}+(y_1-1)z_i\mf y,z\in E,\,i=1,2
\end{equation}
and
$$J=(J_1,J_2),$$
then for $z\in E$, we get
\begin{equation}
\label{E1.24}
(z_1+z_2)\int_E f(y')\,\nu_z(dy')\;=\;\int_Ef\big(z+J(y,z)\big)\,\nu(dy).
\end{equation}
This motivates the following definitions.
Define the functions $I_1$, $I_2$ and $I:=I_1+I_2$ that will serve as intensities for the driving noise by
\begin{equation}
\label{E1.25}
I_{i}(x;k):=\1_{\{x_{3-i}(k)>0\}}\frac{\CA x_{i}(k)}{x_{3-i}(k)}\mf x\in\LBE,\,k\in S,\;t\geq 0,\;i=1,2.
\end{equation}

Let $x\in\LBE$. A pair $(\CN,X)$ is called a weak solution of the following system of stochastic integral equations (for $t\geq0$, $k\in S$, $i=1,2$)
\begin{equation}
\label{E1.26}
 X_{i,t}(k)=x_i(k)+\int_{0}^{t}\CA X_{i,s}(k)\,ds +
\int_{0}^{t}\int_{[0,\infty)\times E}J_i\big(y,X_{s-}(k)\big)\1_{[0,I(X_{s-};k)]}(a)\,\CM\big(\{k\},ds,d(a,y)\big)
\end{equation}
if $\CN$ is a PPP described in \equ{E1.21} and $X$ is an $\FF$-adapted $D_{\LBE}$ valued process such that \equ{E1.26} holds for all $t\geq0$ and $k\in S$.  (Note that $\R_+\times E$ plays the r{\^o}le of $F$ in the considerations around \equ{E1.16} and that $(a,y)$ plays the r{\^o}le of $x$ there.) We say that the solution is unique if the distribution of $X$ is the same for all weak solutions.

In order to grasp the intuitive meaning of \equ{E1.26}, first consider the case $X_{i,s-}(k)>0$. Then Poisson points $y\in E$ come at the rate $(\CA X_{3-i,s-}(k)/X_{i,s-}(k))\nu(dy)$. The point $y$ is turned into a jump $X_s(k)-X_{s-}(k)$ of size $J(y,X_{s-}(k))$.  According to \equ{E1.24} this means that jumps from $X_{s-}(k)$ to some $y'\in E$ come at a rate $\CA X_{3-i,s-}(k)\,\nu_{X_{s-}(k)}(dy')$ as desired. A similar reasoning holds for the case $X_{3-i,s-}(k)>0$.

Note that the stochastic integral  in \equ{E1.26}
\begin{equation}
\label{E1.27}
\MM_{i,t}(k):= \int_{0}^{t}\int_{[0,\infty)\times E}J_i\big(y,X_{s-}(k)\big)\1_{[0,I(X_{s-};k)]}(a)\,\CM\big(\{k\},ds,d(a,y)\big)
\end{equation}
does not make sense as a Lebesgue-Stieltjes integral but is understood in the sense explained around \equ{E1.17} with $x$ replaced by $(a,y)$ and $Z_s(\omega;x)$ replaced by $J\big(y,X_{s-}(k)(\omega)\big)\1_{[0,I_k(X_{s-}(\omega))]}(a)$. Furthermore, note that the left limit $X_{s-}$ is used so as to make the integrand predictable.
In order that the integral be well defined, with a view to \equ{E1.18}, it is enough to check that for some $p\in(1,2)$, we have
\begin{equation}
\label{E1.28}
\E\left[\int_0^t\int_E\big|J_i\big(y,X_{s-}(k)\big)\big|^pI(X_{s-};k)\,ds\,\nu(dy)\right]<\infty.
\end{equation}

In fact, if condition \equ{E1.28} holds for all $p\in(1,2)$, then $(\MM_{i,t}(k))_{t\geq0}$, $i=1,2$, $k\in S$, is an $L^p$-martingale for any $p\in[1,2)$.
We will show below in Remark~\ref{R3.6} that \equ{E1.28} indeed holds for all $p\in(1,2)$.

Note that the whole business of defining $J$ is used in order to define the dynamics of $X$ in terms of only one source of noise that produces ``standard jumps''. If $\CM(\{k\},\ARG,[0,I_k(X_{s-})]\times\ARG)$ has an atom at $\{s\}\times\{y\}$, then the actual jump of $X(k)$ at time $s$ is from $X_{s-}(k)$ to
$$X_{s}(k)= X_{s-}(k)+J(y,X_{s-}(k))
=\cases{
(y_1,y_2)\,X_{1,s-}(k),\mfalls X_{1,s-}(k)>0,\\[2mm]
(y_2,y_1)\,X_{2,s-}(k),\mfalls X_{2,s-}(k)>0.}
$$
There are two types of jumps. If $y_2>0$, then the coordinate $X_{s-}(k)$ changes its type. Since $\nu(\{y:\,y_2>0\})=2/\pi<\infty$ (see Lemma~\ref{LA.1}), the changes of type come at a finite rate as long as $X_s(k)$ is bounded away from $0$.
On the other hand, if $y_1>0$, then the jumps change the size of the population at site $k$ by a \emph{factor} of $y_1$, but not its type.   From the definition of the measure $\nu$ the jumps for which $|y_1-1|$ is small come at an infinite rate.
\begin{theorem}
\label{T1.3}
For any $x\in\LBE$, there exists a unique weak solution $(\CN,X)$ of  \equ{E1.26} and $X$ solves \equ{MP1}.
\end{theorem}

The construction of the solution of \equ{E1.26} requires a lot of effort, including an involved approximation scheme. If we were interested only in the existence of solutions of the martingale problem \equ{MP1}, we could follow an easier route by using the Trotter product approach as performed in \cite{KO}.

It is natural to ask whether the coordinate processes of the solution of \equ{E1.26} ever hit the point $(0,0)$. We conjecture that this is not the case. In fact, for a similar model, with $S$ being a singleton, this is proved in \cite[Thm.{} 1.7]{KM1}.

\subsubsection*{Convergence as the rates go to infinity}
Now let us go back to the Dawson-Perkins model. We would like to clarify our initial motivation that
the process described in Theorems~\ref{T1.1} and \ref{T1.3} is indeed the limit of Dawson-Perkins
process as $\gamma\rightarrow \infty$. Let $Y^{\gamma}=(Y^{\gamma}_1, Y^{\gamma}_2)$ be a solution of \equ{E1.1} with $Y_0\in \LBE$.  This
process with our slightly relaxed assumptions on $\CA$ can be constructed in a way similar to the construction of Dawson and Perkins (see also~\cite{bib:cdg04}). Furthermore, let $X$ be a solution of \equ{MP1} with $X_0=Y_0$.

Clearly, the continuous processes $Y^\gamma$ cannot converge to the
discontinuous process $X$ in the Skorohod topology on $D_{\LBE}$. Hence, in order
to get a limit theorem, we use the weaker Meyer-Zheng ``pseudo-path topology'' (see
\cite{bib:mz84}). Roughly speaking, convergence in the ``pseudo-path topology'' means
convergence for Lebesgue almost all time points. More precisely, for any $f\in
D_{\LBT}$ let $\psi(f)$ denote the image measure on $[0,\infty)\times\LBT$
of $e^{-t}dt$ under the map $t\mapsto(t,f(t))$. Note that $\psi$ is injective
and hence weak convergence in the space of probability measures on
$[0,\infty)\times \LBT$ defines a notion of convergence on $D_{\LBT}$ that is called the
``pseudo-path topology'' by Meyer and Zheng \cite{bib:mz84}.

For the convergence of $Y^\gamma$ to $X$, it is not crucial that in \equ{E1.1} the noise term has the special form of a product. In fact, it is only necessary that the noise is isotropic, strictly positive in $(0,\infty)^2$ and vanishing at the boundary in such a way that it admits a solution with each coordinate nonnegative. Hence, consider the equation
\begin{equation}
\label{E1.29}
\begin{aligned}
Y_{i,t}(k)\;&=\;Y_{i,0}(k)+\int_0^t \sum_{l\in S} \CA(k,l) Y_{i,s}(l)\,ds +\int_0^t \gamma^{1/2}\,\sigma( Y_{s}(k))\,dW_{i,s}(k),\;\;t\geq 0,\; k\in S,\,i=1,2.
\end{aligned}
\end{equation}
Here $(W_i(k),\; k\in S,\,i=1,2)$ is an independent family of one-dimensional Brownian motions and
$\sigma:[0,\infty)^2\to[0,\infty)$ is measurable and fulfils the following assumptions:
\begin{assumption}
\label{A1.4}
\begin{itemize}
\item[(i)] $\sigma(x)=0$ for all $x\in E$.
\item[(ii)] $\inf\sigma(C)>0$ for any compact $C\subset(0,\infty)^2$.
\item[(iii)]
For each $y\in\LBT$ and $\gamma>0$, \equ{E1.29} admits a (weak) $\LBT$-valued solution.
\end{itemize}
\end{assumption}
Of course, $\sigma(x)=\sqrt{x_1x_2}$ is the case considered in \equ{E1.1} and it satisfies the above assumptions.

\begin{theorem}
\label{T1.5}
Assume that (i) and (ii) hold and that for each $\gamma>0$, we have chosen an $\LBT$-valued solution $Y^{\gamma}$ of \equ{E1.29}. Assume that $X_0:=Y_0^{\gamma}\in\LBE$ does not depend on $\gamma$.
Then, for each sequence $\gamma_n\to\infty$, in $D_{\LBT}$ equipped with the Meyer-Zheng pseudo-path topology, we have the convergence in law
$$
Y^{\gamma_n} \Longrightarrow X\quad\mbsr{as}n\to\infty.
$$
\end{theorem}

\subsection{Organization of the paper}
\label{S1.4}
We prove the existence parts of Theorems~\ref{T1.1} and \ref{T1.3} via an approximation procedure. In Section~\ref{S2}, we will construct
a family $(\appX)_{m\in\N,\,\ve>0}$ of processes which
\begin{itemize}
\item
live on finite site spaces $S_m\subset S$,
\item
have a finite jump measure $\nu^\ve$ instead $\nu$ by suppressing certain small jumps, and
\item
where the intensities of the driving noise (see \equ{E1.25}) are truncated for small values $x_i(k)\in(0,\ve)$.
\end{itemize}
In that section, we further derive moment estimates for the truncated measures and processes.

In Section~\ref{S3} we show that the sequence
$(\appX)_{m\in\N,\,\ve>0}$ is tight and that any (weak) limit point solves the martingale problem \equ{MP1}. Sections~\ref{S2} and \ref{S3} are the corner stone for the existence part in Theorem~\ref{T1.3}.
In Section~\ref{S4}, we will show uniqueness for \equ{MP1}. Section~\ref{S5} is devoted to the proof of
Theorem~\ref{T1.3} based on the Section~\ref{S3} and \ref{S4}. Theorem~\ref{T1.5} is proved in Section~\ref{S6}.

\section{Approximating processes and moment bounds}
\label{S2}
\setcounter{equation}{0}
\setcounter{theorem}{0}

The aim of this section is to construct
the family of
approximating processes $(\appX)$, that was announced in Section~\ref{S1.4}. To this end define a sequence $(S_m)_{m\in\N}$ of \emph{finite} subsets of $S$ such that $S_m\uparrow S$
as $m\rightarrow \infty$. The process $\appX$ formally lives on $S$ but we keep fixed all coordinates in $S\setminus S_m$.

\subsection{Definition of the approximating processes}
\label{S2.1}
We will define a family of approximating processes
\[\appX=\big((\appX_{1,t}(k), \appX_{2,t}(k))\in E,\; k\in S,\, t\geq 0\big)\]
in a way that they may change values only for $k\in S_m$ and stay constant for $k\in S\setminus S_m$.
To this end let us
define the matrix $\apmCA$ by
$$\apmCA(k,l)=\cases{
  \CA(k,l),&\mfalls k,l\in S_m,\\
  0,&\msonst.
}
$$
Let $(\apmCS_t)_{t\geq 0}$ be the semigroup generated by $\apmCA$ and let $\apmp_t=e^{t\apmCA}$ denote its kernel, that is, for $f\in\LB$
$$
\apmCS_t f(k)= \sum_{l\in S} \apmp_t(k,l)f(l).
$$
Define $\apmA(k,l)=\apmCA(k,l)^+$ and let $(\apmbS_t)_{t\geq0}$ be the semigroup generated by $\apmA$.
Clearly, for any $f\in \LB$,
\begin{equation}
\label{E2.1}
\begin{aligned}
\apmCA f(k)&\;\leq\; \apmA f(k)\;\leq\; \bA f(k)\mfa k\in S, \\[1mm]
\apmCS_t f(k)&\;\leq\; \apmbS_t f(k)\;\leq\; \bS_t f(k)\mfa k\in S.
\end{aligned}
\end{equation}

We denote by $a\vee b:=\max(a,b)$ the maximum and by $a\wedge b:=\min(a,b)$ the minimum of two numbers. Fix $\ep\in(0,1)$ and $m\in\N$ and define the modified jump rate (compare \equ{E1.25})
\begin{equation}
\label{E2.2}
\app{I}_i(x;k):=\1_{\{x_{3-i}(k)>0\}}\,\frac{\apmCA x_{i}(k)}{x_{3-i}(k)\vee\ve}\mf i=1,2,
\end{equation}
and
$$\app I(x;k)=\app I_1(x;k)+\app I_2(x;k).$$
For $y,z\in E$ and $i=1,2$, define (compare \equ{E1.23})
\begin{equation}
\label{E2.3}
\app J_{i}(y,z)=y_2(z_{3-i}\vee\ve)\1_{\{z_{3-i}>0\}}+(y_1-1)z_i
\end{equation}
and
$$\app J=\big(\app J_1,\app J_2\big).$$
Note that these definitions are pretty much in line with the definitions of $I$ and $J$ in \equ{E1.25} and \equ{E1.23}, but here small positive values of $z_i$ are replaced by $\ve$. This handles the problem of increasing jump rates when a coordinate approaches $0$.

Next, we take care of the problem that the jump measure $\nu$ is infinite.
We introduce the (finite) truncated jump measure $\nu^\ve:=\nu\1_{E^\ve}$ where
$$E^\ve:=\big\{y\in E:\,y_1\not\in(1-\ve,1+\ve')\big\},$$
that is,
\begin{equation}
\label{E2.4}
\nu^\ve(dy)=\nu(dy)\1_{\{y_1\not\in(1-\ve,1+\ve')\}}.
\end{equation}
Here, $\ve':=\ve'(\ve)\in[\ve/2,\ve]$ is chosen according to Lemma~\ref{LA.3} such that
$$\int_{E^\ve}(y_1-1)\,\nu(dy)=0.$$
This particular form of the truncated jump measure is helpful since it preserves the expectation of jumps.

Note that
$$
\int_{E}\app J_{i}(y,z)\,\nu(dy)=(z_{3-i}\vee\ve)\,\1_{\{z_{3-i}>0\}},
$$
$$
\hspace*{-1.3mm}
\int_{E}\app J_{i}(y,z)\,\nu^\ve(dy)=(z_{3-i}\vee\ve)\,\1_{\{z_{3-i}>0\}},
$$
and
\begin{equation}
\label{E2.5}
\hspace*{-14mm}\app I(x;k)\,\int_{E}\app J_{i}(y,x)\,\nu^\ve(dy)=\apmCA x_i(k)\,\1_{\{x_{3-i}(k)>0\}}.
\end{equation}

For the approximating process (but not for the limiting process $X$), we have to take special care of the coordinates that assume the value $0$. If for a given coordinate $k$, we have $x(k)=(0,0)$, then the drift $\apmCA x(k)$ would drive the process out of the state space $\LBE$ immediately. (This would happen also to a coordinate with $x_i(k)>0$ if we would impose the deterministic drift $\CA x_{3-i}(k)>0$ instead of the jump process with this compensator.) This shows that we have to replace the deterministic drift by a jump process whose compensator is given by $\apmCA x(k)$. There are several ways to do so, for example, one could use for each $i=1,2$, a Poisson process with jump size $\ve$ and rate $\ve^{-1}\apmCA x_{3-i}(k)$. Here, in order to stick formally with the noise processes $\CN$ defined in \equ{E1.21} (and for no other reason), we define two independent (and independent of $\CN$) noises $\CN^1$ and $\CN^2$ with the same distribution as $\CN$ and let $\CM^i:=\CN
 ^i-\CN'$ denote the compensated jump measure, $i=1,2$. Finally, assume that $(\CF_t)_{t\geq0}$ is the filtration generated by $\CN$, $\CN^1$ and $\CN^2$ and that fulfils the usual conditions.

The intensities of the jumps away from $0$ will be given by
\begin{equation}
\label{E2.6}
\appIN_i(x;k):=\1_{\{x(k)=(0,0)\}}\,\frac{1}{\ve}\apmCA x_{i}(k).
\end{equation}
Note that
\begin{equation}
\label{E2.7}
\int_E\ve y_2\appIN_i(x;k)\,\nu(dy)=\1_{\{x(k)=(0,0)\}}\,\apmCA x_{i}(k).
\end{equation}

Now given a process $\appX$
which is adapted to the filtration
$(\CF_t)_{t\geq0}$, define
$\Delta \appX_s=\appX_s-\appX_{s-}$ the jump of $\appX$ at time $s$.

Now we can define the process $\appX$ as the unique strong solution of the system of equations
\begin{equation}
\label{E2.8}
\begin{aligned}
\appX_{i,t}(k)\,=\,x_i(k)&+\int_{0}^{t}\int_{[0,\infty)\times E^\ve} \app J_{i}\big(y,\appX_{s-}(k)\big)\1_{[0,\app I(\appX_{s-};k)]}(a)\,\CN\big(\{k\},ds,d(a,y)\big)\\
&+\int_0^t\int_{[0,\infty)\times E} \ve y_2\,\1_{[0,\appIN_i(\appX_{s-};k)]}(a)\,\CN^i\big(\{k\},ds,d(a,y)\big)\\
&+\int_0^t\apmCA\appX_{i,s}(k)\,\1_{\{\appX_{i,s}(k)>0\}}\,ds
,\mf k\in S_m,\; i=1,2\end{aligned}
\end{equation}
and
\begin{equation}
\label{E2.9}
\appX_t(k)=x(k)\mf t\geq0\mbs{and}k\in S\setminus S_m.
\end{equation}
Note that the middle term in \equ{E2.8} represents the jumps in the case where $\appX_{s-}(k)=(0,0)$. For nonzero coordinates this term is zero. We will show later that in the limit $\ve\downarrow0$, the compensated version of this term vanishes.

Note that $\app J$ is defined in a way that the jumps indeed do not drive the coordinate processes out of the space $E$; that is, $z+\app J(y,z)\in E$ for all $y,z\in E$. Also note that the middle term in \equ{E2.8} does not drive the coordinates out of $E$ since it is nonzero only if the coordinate takes the value $(0,0)$ and in this case the value of only one type changes by jump.

Since the jumps according to $\nu^\ve$ have a finite mean, the total mass process increases at most exponentially with the number of jumps that occur. Since the jump rate of $\appX$ is bounded by  $\nu(E^\ve)\,\ve^{-1}|S_m|$ times the total mass, in each time interval there are in fact at most finitely many jumps. Hence the solution $\appX$ of \equ{E2.8} and \equ{E2.9} is indeed well defined and is unique.

We want to write the dynamics of $\appX$ as a sum of the ``heat flow'' and the martingale term of compensated jumps. To this end, we define
the martingale
\begin{equation}
\begin{aligned}
\label{E2.10}
\app\MM_{i,t}(k):=&
\int_{0}^{t}\int_{[0,\infty)\times E^\ve} \app J_i\big(y,\appX_{s-}(k)\big)\1_{[0,\app I(\appX_{s-};k)]}(a)\,\CM\big(\{k\},ds,d(a,y)\big)\\
&+\int_0^t\int_{[0,\infty)\times E} \ve y_2\,\1_{[0,\appIN_i(\appX_{s-};k)]}(a)\CM^i\big(\{k\},ds,d(a,y)\big).
\end{aligned}
\end{equation}
Hence, by subtracting and adding the compensator terms in~\equ{E2.8}, we can rewrite~\equ{E2.8} as
\begin{equation}
\label{E2.8a}
\begin{aligned}
\appX_{i,t}(k)\,=\,x_i(k)&+\app\MM_{i,t}(k)\\
&+ \int_{0}^{t}\int_{E^\ve}  \app J_i\big(y,\appX_{s}(k)\big)\,\app I\big(\appX_{s};k\big)\,\nu(dy)\,ds\\
&+\int_{0}^{t}\int_{E}  \ve y_2\,\appIN_i\big(\appX_{s};k\big)\,\nu(dy)\,ds\\
&+\int_0^t\apmCA\appX_{i,s}(k)\,\1_{\{\appX_{i,s}(k)>0\}}\,ds
,\mf k\in S_m,\; i=1,2.\end{aligned}
\end{equation}

Adding the last three terms in \equ{E2.8a}, we get (using \equ{E2.5} and \equ{E2.7} in the first equality)
$$
\begin{aligned}
\,\int_{0}^{t}&\int_{E^\ve}  \app J_i\big(y,\appX_{s}(k)\big)\,\app I\big(\appX_{s};k\big)\,\nu(dy)\,ds\\
&\phantom{-}+\int_{0}^{t}\int_{E}  \ve y_2\,\appIN_i\big(\appX_{s};k\big)\,\nu(dy)\,ds\\
&\phantom{-}+
\int_0^t\apmCA\appX_{i,s}(k)\1_{\{\appX_{i,s}(k)>0\}}\,ds\\
=&\phantom{-}\int_0^t\apmCA\appX_{i,s}(k)
\Big(\1_{\{\appX_{3-i,s}(k)>0\}}+
\1_{\{\appX_{s}(k)=(0,0)\}}
+\1_{\{\appX_{i,s}(k)>0\}}\Big)\,ds\\
=&\;\int_0^t\apmCA\appX_{i,s}(k)\,ds.\end{aligned}
$$

Now, \equ{E2.8a} can be rewritten as
\begin{equation}
\label{E2.11}
\appX_{i,t}(k)=x_i(k)+\int_{0}^{t}\apmCA \appX_{i,s}(k)\,ds\;+ \;\app\MM_{i,t}(k) ,\quad k\in S_m,\; i=1,2.
\end{equation}

\begin{remark}
\label{R2.1}Our way of defining the approximate processes might look a bit special at first glance. However, the more naive idea of truncating the rates $I$ like $\min(I,M)$ for some $M>0$ and then letting $M\to\infty$ along with $\ve\downarrow0$ did not work since this results in an additional drift term in \equ{E2.11} that we could not control.

The other idea that naturally pops up is to suppress the jumps of small size on the absolute scale; that is, suppressing all jumps where $|\Delta X_{t}(k)|>\ve$ for some $\ve$. However also in this case, we could not control the additional error term.
\end{remark}

\subsection{Moment estimates for the approximate processes}
\label{S2.3}
The aim is to let  $\ve\downarrow0$ and $m\to\infty$ and to show that $\appX$ converges to a solution of \equ{E1.26}. To this end, we need moment bounds on $\appX$ that are uniform in $m$ and $\ve$ (Corollary~\ref{C2.3}) and that will be derived by the following martingale decomposition for the product $\appX_{1,t}(k)\appX_{2,t}(k)$.
\begin{lemma}
\label{L2.2}
Let $k_1,k_2\in S$, $k_1\neq k_2$, and let $t>0$. Then
\begin{equation*}
\begin{aligned}
\appX_{i,t}(k_i)=
\apmCS_t
\appX_{i,0}(k_i)
+\sum_{l\in S}\int_0^t \apmp_{t-s}(k_i,l)\,d\app{\MM}_{i,s}(l)
\end{aligned}
\end{equation*}
and
\begin{equation*}
\begin{aligned}
\appX_{1,t}(k_1)&\appX_{2,t}(k_2) =
\apmCS_t
\appX_{1,0}(k_1)
\apmCS_t
\appX_{2,0}(k_2)\\
&-\sum_{l\in S}\int_0^t \apmp_{t-s}(k_1,l)\apmp_{t-s}(k_2,l) \left(
\appX_{1,s}(l)\apmCA \appX_{2,s}(l) + \appX_{2,s}(l)\apmCA \appX_{1,s}(l)\right)\,ds\\[2mm]
&+ \app{M}_t,
\end{aligned}
\end{equation*}
where $\app{M}_t$ is given by
\begin{equation*}
\begin{aligned}
\app{M}_t &= \sum_{\attop{l_1,l_2\in S}{l_1\neq l_2}}\int_0^{t}
 \apmp_{t-s}(k_1,l_1)\apmp_{t-s}(k_2,l_2)\left(\appX_{1,s-}(l_1)\,d\app\MM_{2,s}(l_2)
+\appX_{2,s-}(l_2)\,d\app\MM_{1,s}(l_1)\right).
\end{aligned}
\end{equation*}
\end{lemma}
\paragraph{Proof.}
The first equation is just the mild form of \equ{E2.11}.

The second equality is an easy application of the integration by parts formula and the fact that
 \[\appX_{1,t}(k)\appX_{2,t}(k)=0\mfa t\geq 0,\, k\in S.\eope\]

\begin{corollary}
\label{C2.3}
Let $k_1,k_2\in S$ and $t>0$. Then we have
$$
\E\left[ \appX_{i,t}(k_i)\right]\leq \bS_t\,x_i(k_i)
$$
and
$$
 \E\left[ \appX_{1,t}(k_1)\appX_{2,t}(k_2)\right] \;\leq\; \bS_t\,x_1(k_1) \,\bS_t\,x_2(k_2).$$
\end{corollary}
\begin{Proof}
Consider first the case where $k_1,k_2\in S_m$.
We have
$$
\appX_{1,t}(k_1)\appX_{2,t}(k_2) \leq
\apmCS_t
\appX_{1,0}(k_1)
\apmCS_t
\appX_{2,0}(k_2)+M_t.$$
For $K>0$ define the stopping time $\tau_K:=\inf\{s\geq0:\sum_{l}\sum_{i=1}^2\appX_{s,i}(l)\geq K\}\wedge t$.
A simple stopping argument shows that
$$
\begin{aligned}
\apmCS_{t-\tau_K}X_{1,\tau_K}(k_1)&\,\apmCS_{t-\tau_K}X_{2,\tau_K}(k_2)
\leq
\apmCS_t
\appX_{1,0}(k_1)\,
\apmCS_t\appX_{2,0}(k_2)
\\
&+
\sum_{\attop{l_1,l_2\in S}{l_1\neq l_2}}\int_0^{\tau_K}
 \apmp_{t-s}(k_1,l_1)\apmp_{t-s}(k_2,l_2)\left(\appX_{1,s-}(l_1)\,d\app\MM_{2,s}(l_2)
+\appX_{2,s-}(l_2)\,d\app\MM_{1,s}(l_1)\right).
\end{aligned}
$$

Since now the integrand is bounded and the integrators are martingales, the integral has expectation zero, and we get
$$\E\Big[\apmCS_{t-\tau_K}X_{1,\tau_K}(k_1)\,\apmCS_{t-\tau_K}X_{2,\tau_K}(k_2)\Big]
\;\leq\;
\apmCS_t
\appX_{1,0}(k_1)
\,\apmCS_t\appX_{2,0}(k_2).
$$
Since $\P[\tau_K=t]\To{K\to\infty}1$, the expression in the expectation tends to $\appX_{1,t}(k_1)\appX_{2,t}(k_2)$. Using Fatou's lemma, we conclude
\[\E\big[\appX_{1,t}(k_1)\appX_{2,t}(k_2)\big] \;\leq\;
\apmCS_{t}
\appX_{1,0}(k_1)
\,\apmCS_{t}
\appX_{2,0}(k_2).\eope\]
\end{Proof}

Now we will give the $L^p$-bounds for the martingales $\app\MM_i, i=1,2$.
\begin{lemma}
\label{L2.4}
For any  $p\in (1,2)$, there exists a constant $c_p<\infty$ such that for all $m\in\N$, $\ve>0$, $k\in S$, $T>0$ and $i=1,2$, we have
$$
\E\left[ \sup_{t\leq T} | \app\MM_{i,t}(k)|^p\right]
\leq c_p\int_0^T  \left[ (\bA+1) \bS_s x_1(k)+1\right]
 \left[(\bA+1)\bS_s x_2(k)+1\right]\,ds.
$$
\end{lemma}
\paragraph{Proof.}
First  note that for $k\not\in S_m$, we have $\app\MM_i(k)\equiv0$ and hence the estimate is trivial.

Now let $k\in S_m$. Recall the definitions of $\app \MM$ and $J$ in \equ{E2.10} and \equ{E1.23}, respectively, and note that
$$\begin{aligned}
\E\left[ \sup_{t\leq T} | \app\MM_{i,t}(k)|^p\right]
&\leq
3^{p-1} \E\left[ \sup_{t\leq T} \left|C_t\right|^p\right]
+3^{p-1} \E\left[ \sup_{t\leq T} \left|D_t\right|^p\right]
+3^{p-1} \E\left[ \sup_{t\leq T} \left|E_t\right|^p\right]\\
\nonumber
&=: L_1+L_2+L_3,
\end{aligned}
$$
where
$$\begin{aligned}
C_t&:=\int_{0}^{t}\int_{[0,\infty)\times E^\ve} y_2\big(\appX_{3-i,s-}(k)\vee\ve\big)\1_{[0,\app I_{i}(\appX_{s-};k)]}(a)\,\CM\big(\{k\},ds,d(a,y)\big),\\
D_t&:=\int_{0}^{t}\int_{[0,\infty)\times E^\ve} (y_1-1)\appX_{i,s-}(k)\1_{[0,\app I_{3-i}(\appX_{s-};k)]}(a)\,\CM\big(\{k\},ds,d(a,y)\big),
\end{aligned}
$$
and
$$\hspace*{-22.1mm}E_t:=\int_0^t\int_{[0,\infty)\times E} \ve y_2\,\1_{[0,\appIN_i(\appX_{s-};k)]}(a)\CM^i\big(\{k\},ds,d(a,y)\big)
$$
are martingales of finite variation.
As the point process $\CN$ has no double points, the square variation process of $C$ is
$$[C,C]_t=\int_{0}^{t}\int_{[0,\infty)\times E^\ve} y_2^2\big(\appX_{3-i,s-}(k)\vee\ve\big)^2\1_{[0,\app I_{i}(\appX_{s-};k)]}(a)\,\CN\big(\{k\},ds,d(a,y)\big)
$$
Let $m_{1,p}:=\int_E|y_1-1|^p\,\nu(dy)$ and $m_{2,p}:=\int_Ey_2^p\,\nu(dy)$ denote the $p$-th moments of $\nu$. By Lemma~\ref{LA.5}, we have that both quantities are finite. Hence, by the Burkholder-Gundy-Davis inequality (see, e.g., \cite[Theorem VII.92]{DellacherieMeyer1983.5-8}) we get with $c_p'=3^{p-1}(4p)^p$
\begin{equation}
\label{E2.12}
\begin{aligned}
L_1 &\leq c_p'\,\E\big[[C,C]_t^{p/2}\big]\\
&\leq c_p'\, \E\left[\int_{0}^{t}\int_{[0,\infty)\times E^\ve} y_2^p\big(\appX_{3-i,s-}(k)\vee\ve\big)^p\1_{[0,\app I_{i}(\appX_{s-};k)]}(a)\,\CN\big(\{k\},ds,d(a,y)\big)\right]
\\
&= c_p'\, \E\left[\int_{0}^{t}\int_{[0,\infty)\times E^\ve} y_2^p\big(\appX_{3-i,s-}(k)\vee\ve\big)^p\1_{[0,\app I_{i}(\appX_{s};k)]}(a)\,\CN'\big(\{k\},ds,d(a,y)\big)\right]
\\
&= c_p'\, \E\left[\int_{0}^{t}\int_{E} y_2^p\big(\appX_{3-i,s}(k)\vee\ve\big)^p\app I_{i}(\appX_{s};k)\,\nu(dy)\,ds\right]
\\
&\leq c_p'\,m_{2,p}\, \E\left[\int_0^T \big(\appX_{3-i,s}(k)\vee\ve\big)^{p-1}\1_{\{\appX_{3-i,s}(k)>0\}}\, \apmCA\appX_{i,s}(k) \,ds\right]\\
&\leq c_p'\,m_{2,p}\, \E\left[\int_0^T \big(\appX_{3-i,s}(k)+1\big)\1_{\{\appX_{3-i,s}(k)>0\}} \apmCA\appX_{i,s}(k) \,ds\right]\\
&\leq c_p'\,m_{2,p}\, \int_0^T \big(\bS_s x_{3-i}(k)+1\big) \bA \bS_s x_i(k) \,ds
\end{aligned}
\end{equation}
where the last inequality follows  by Corollary~\ref{C2.3} and \equ{E2.1}.
The last line of \equ{E2.12} is trivially bounded by
\begin{equation}
\label{E2.13}
c_p'\,m_{2,p}\,\int_0^T  \left[ (\bA+1) \bS_s x_1(k)+1\right]
 \left [ (\bA+1)\bS_s x_2(k)+1\right]\,ds.
\end{equation}
Hence we are done for $L_1$.

Similarly, we get for $L_3$ (by taking $\appX_{3-i,s}(k)=0$) that
$$L_3\leq \ve^{p-1}
c_p'\,m_{2,p}\, \int_0^T  \bA \bS_s x_i(k) \,ds.$$

Now we treat the $L_2$ term.
Again, by the Burkholder-Gundy-Davis inequality, we get
\begin{equation}
\label{E2.14}
\begin{aligned}
L_2
&\leq c_p'\, \E\left[ \left(\int_0^T \int_{[0,\infty)\times E^\ve}
(y_1-1)^2\appX_{i,s-}(k)^2\1_{[0,\app I_{3-i}(\appX_{s-};k)]}(a)\,\CN\big(\{k\},ds,d(a,y)\big) \right)^{p/2}\right]\\
&\leq c_p'\, \E\left[ \int_0^T \int_{[0,\infty)\times E}
|y_1-1|^p\,\appX_{i,s-}(k)^p\,\1_{[0,\app I_{3-i}(\appX_{s-};k)]}(a)\,\CN\big(\{k\},ds,d(a,y)\big) \right]\\
& =c_p'\, \E\left[ \int_0^T \int_{E}
|y_1-1|^p\,\appX_{i,s}(k)^p\app I_{3-i}(\appX_{s};k)\,\nu(dy)\,ds \right]\\
&\leq c_p'\,m_{1,p}\, \E\left[\int_0^T \big(\appX_{i,s}(k)\big)^{p-1} \apmCA\appX_{3-i,s}(k) \,ds\right]\\
&\leq c_p'\,m_{1,p}\,\E\left[\int_0^T \big(\appX_{i,s}(k)+1\big) \1_{\{\appX_{i,s}(k)>0\}}\apmCA\appX_{3-i,s}(k) \,ds\right]\\
&\leq c_p'\,m_{1,p}\,\int_0^T \big(\bS_s x_i(k)+1\big) \bA\bS_s x_{3-i}(k) \,ds,
\end{aligned}
\end{equation}
where the last inequality follows  by Corollary~\ref{C2.3} and \equ{E2.1}.
Again, the right hand side of \equ{E2.14}
is trivially bounded by \equ{E2.13}. Now the claim holds with $c_p=c_p'(m_{1,p}+2m_{2,p})$ (which is finite by Lemma~\ref{LA.5}).
\gdm

\begin{remark}
\label{R2.5}
Note that the bound in the above lemma is uniform in $m$ and $\ve$.
\end{remark}
From Lemma~\ref{L2.4}, it is easy to derive the following bound (uniform in $m$) on the moments of  the increments of $\appX_i(k)$.
\begin{lemma}
\label{L2.6}
For any $r_1,r_2\in (0,1]$ such that $1<r_1/r_2<2$, there exists a constant $c=c(r_1,r_2)$ such that for all $T>0$, $k\in S$ and $i=1,2$, we have
$$
\E\left[ \sup_{t\leq T} \left|\appX_{i,t}(k)-x_i(k)\right|^{r_1}\right] \leq c \,\max_{j=1,2}
\left(\int_0^T\prod_{i'=1}^2 \big[(\bA+1) \bS_sx_{i'}(k)+1\big]
\,ds\right)^{r_j}.
$$\end{lemma}
\begin{proof}
For $k\in S\setminus S_m$, the result is trivial. Hence now let $k\in S_m$.
By equation \equ{E2.8} and the triangle inequality, we get
$$
\E\left[\sup_{t\leq T} \left|\appX_{i,t}(k)-x_i(k)\right|^{r_1}\right]
\leq R_1+R_2,
$$
where
$$
R_1=\E\left[
\sup_{t\leq T} \left(\int_{0}^{t} \bA \appX_{i,s}(k)\,ds\right)^{r_1}\right] \mbu
R_2= \E\left[\sup_{t\leq T}\left|\app\MM_{i,t}(k)\right|^{r_1}\right].
$$
As in $R_1$ the integrand is nonnegative, and using Jensen's inequality and \equ{E2.17},
we get
\begin{equation}
\label{E2.15}
R_1\leq  \bigg(\E\bigg[\int_{0}^{T} \bA \apmX_{i,s}(k)\,ds\bigg]\bigg)^{r_1}
\leq \bigg(\int_{0}^{T} \bA \bS_s x_i(k)\,ds\bigg)^{r_1}.
\end{equation}
For $R_2$, by Jensen's inequality and Lemma~\ref{L2.4} (with $p=r_1/r_2$), we get that for some constant $c_{r_1/r_2}<\infty$,
\begin{equation}
\label{E2.16}
R_2\leq \bigg(\E\bigg[\sup_{t\leq T}\left|\ap\MM_{i,t}(k)\right|^{r_1/r_2}\bigg]\bigg)^{r_2}\leq c_{r_1/r_2}\bigg(\int_0^T  \prod_{j=1}^2 \big[(\bA+1) \bS_sx_j(k)+1\big]\,ds\bigg)^{r_2}.
\end{equation}
Combining \equ{E2.15} and \equ{E2.16} gives the claim of this lemma.
\end{proof}

First, by Corollary~\ref{C2.3}, we have
\begin{equation}
\label{E2.17}
\E\left[ \appX_{i,t}(k)\right] \leq  \bS_t\,x_i(k)\mf k\in S, \; i=1,2,
\end{equation}
and hence by \equ{E1.11},
$$
\E\left[\big\langle \appX_{i,t},\beta\big\rangle\right]
\leq e^{\LSC t}\,\langle x_i,\beta\rangle \mf i=1,2.
$$
Now we derive bounds on $\sup_{t\leq T}\big\langle \appX_{i,t},\beta\big\rangle, \;i=1,2$.
\begin{lemma}
\label{L2.7}
Let $\phi$ be a non-negative function on $S$. For any $T,K>0$,
$$
\P\left[\sup_{t\leq T} \left\langle \appX_{i,t}\,,\phi\right\rangle >K\right]
\leq K^{-1}\big\langle \bS_T\,x_i,\phi\big\rangle.
$$
In particular, for  $x\in\LBE$, we have (recall $\beta$ and $\LSC$ from \equ{E1.5})
$$
\P\left[\sup_{t\leq T} \left\langle \appX_{i,t}\,,\beta\right\rangle >K\right]
\leq K^{-1}e^{\LSC T}\|x_i\|_\beta.
$$
\end{lemma}
\paragraph{Proof.} First assume that $\phi$ has finite support.
Recall $\app{\MM}$ from \equ{E2.11}. Define the submartingale
$$
\app{M}_{i,t}(k):=x_i(k)+\app{\MM}_{i,t}(k)+\int_0^t \apmA\appX_{i,s}(k)\,ds
$$
and note that $0\leq \appX_{i,t}(k)\leq \app{M}_{i,t}(k)$.
By Doob's inequality, we get
$$
\P\left[ \sup_{t\leq T}\big\langle \appX_{i,t},\phi\big\rangle >K\right]\;\leq\; \P\left[ \sup_{t\leq T}\big\langle \app{M}_{i,t},\phi\big\rangle >K\right]
\;\leq\;\frac{\E\big[ \big\langle \app{M}_{i,T},\phi\big\rangle\big]}{K}.
$$
By \equ{E2.17}, we get
$$
\E\left[\app{M}_{i,T}(k)\right]\;\leq\; x_i(k)+\int_0^T(\bA\bS_t)x_i(k)\,dt\;=\;\bS_Tx_i(k).
$$
Hence
$$\E\big[ \big\langle \app{M}_{i,T},\phi\big\rangle\big]\,\leq\,\big\langle \bS_Tx_i,\phi\big\rangle,
$$
which finishes the proof for $\phi$ with finite support.
 For general $\phi\in[0,\infty)^S$, the claim follows by monotone convergence.
\gdm

\begin{lemma}
\label{L2.8}
Fix arbitrary $T>0$.
Let $(\tau_m)_{m\in\N}$ be a sequence of stopping times bounded by $T$. Then for any $r_1\in (0,1)$ and $k\in S$, we have
$$
\lim_{\delta\downarrow 0}\sup_{m\in\N,\,\ve>0}  \,\E\left[ \big|\appX_{i,\tau_{m}+\delta}(k)-\appX_{i,\tau_{m}}(k)\big|^{r_1}\right]=0.
$$

\end{lemma}
\paragraph{Proof.}
Without loss of generality, we may assume $\delta\leq1$.
We define the stopping time
\[\sigma_{m,K}=\inf\Big\{t\geq0:\,\big\langle \appX_{1,t}+\appX_{2,t},\,\beta\big\rangle\geq K \Big\} \]
and let
$$
\app{R}_1:=\E\left[ \big|\appX_{i,\tau_{m}+\delta}(k)-\appX_{i,\tau_{m}}(k)\big|^{r_1}\1_{\{\sigma_{m,K}\leq T+1\}}
\right]
$$
and
$$
\app{R}_2:=\E\left[ \big|\appX_{i,\tau_{m}+\delta}(k)-\appX_{i,\tau_{m}}(k)\big|^{r_1}\1_{\{\sigma_{m,K}> T+1\}}
\right].
$$
Let $p>1$ be such that $pr_1\leq 1$ and define $q>1$ by $1/p+1/q=1$. Then by H{\"o}lder's inequality, we have
$$
\begin{aligned}
\app{R}_1&\leq\left( \E\left[ \big|\appX_{i,\tau_{m}+\delta}(k)-\appX_{i,\tau_{m}}(k)\big|^{pr_1}\right]\right)^{1/p}
 \P\left[\sigma_{m,K}\leq T+1
\right]^{1/q}\\
&\leq 2\left( \E\left[\sup_{t\leq T+1} \appX_{i,t}(k)^{pr_1}\right]\right)^{1/p}
\P\left[\sup_{t\leq T+1} \big\langle \appX_{1,t}+\appX_{2,t},\beta\big\rangle\geq K\right]^{1/q}
\end{aligned}
$$
By Lemma~\ref{L2.6}, we have
\[ h(T):=\sup_{m\in\N,\,\ve>0}\,2\left( \E\left[\sup_{t\leq T+1} \appX_{i,t}(k)^{pr_1}\right]\right)^{1/p}
<\infty.\]
Let $\delta_1>0$. By Lemma~\ref{L3.4}, we can choose $K$ sufficiently large such that
\[
\sup_{m\in\N,\,\ve>0}\P\left[\sup_{t\leq T+1} \big\langle \appX_{1,t}+\appX_{2,t},\beta\big\rangle\geq K\right]^{1/q}\leq \frac{\delta_1}{h(T)}.\]
This implies that
\begin{equation}
\label{E2.18}
\sup_{m\in\N,\,\ve>0} \app{R}_1\leq \delta_1\,.
\end{equation}
Now we turn to $\app{R}_2$. Let $r_2\in(r_1/2,r_1)$. By the strong Markov property of $\appX$ and Lemma~\ref{L2.6}, we obtain
\begin{equation}
\label{E2.19}
\begin{aligned}
\E&\left[
\big|\appX_{i,\tau_{m}+\delta}(k)-\appX_{i,\tau_{m}}(k)\big|^{r_1}
\1_{\{\sigma_{m,K}>T+1\}}\right]\\
&\leq
\E\left[
\big|\appX_{i,\tau_{m}+\delta}(k)-\appX_{i,\tau_{m}}(k)\big|^{r_1}\1_{\{\langle \appX_{1,\tau_m}+\appX_{2,\tau_m},\beta\rangle\leq K\}}\right]\\
&=
\E\left[\E_{\appX_{\tau_m}}\big[\big|\appX_{i,\delta}(k)-\appX_{i,0}(k)\big|^{r_1}\big]\,
\1_{\{\langle \appX_{1,\tau_m}+\appX_{2,\tau_m},\beta\rangle\leq K\}}\right]
\\
&\leq c(r_1,r_2) \max_{j=1,2}
\E\bigg[\bigg(\int_0^\delta \prod_{i'=1}^{2} \big[(\bA+1) \bS_s\appX_{i',\tau_{m}}(k)+1\big]\,ds
\bigg)^{r_j}\\
&\mbox{}\hspace*{6cm}\left.\times
\1_{\{\langle \appX_{1,\tau_m}+\appX_{2,\tau_m},\beta\rangle\leq K\}}\right].
\end{aligned}
\end{equation}
Note that on the event $\{\langle \appX_{1,\tau_m}+\appX_{2,\tau_m},\beta\rangle\leq K\}$, by \equ{E1.9}, \equ{E1.9} and \equ{E1.10}, we have
$$\big((\bA+1)\bS_s\appX_{i,\tau_m}\big)(k)\,\leq \,(\LSC+1)\,e^{\LSC s}\,K/\beta(k).$$
Hence for some constant $c(r_1,r_2,\LSC,T)$, the right hand side of \equ{E2.19} is bounded by
$$
\begin{aligned}
 c(r_1,r_2,\LSC,T) \max_{j=1,2}\bigg(\int_0^\delta  \left( \frac{K}{\beta_k}+1\right)^2
\,ds\bigg)^{r_j}\longrightarrow0 \;\;{\rm as}\; \delta\downarrow 0,
\end{aligned}
$$
uniformly in $m$. Together with \equ{E2.18}, this implies
$$
\limsup_{\delta\downarrow 0}\sup_{m\in\N,\,\ve>0}\big(\app{R}_1+ \app{R}_2\big)\leq \delta_1\,.
$$
Since $\delta_1>0$ was arbitrary, the limit is in fact $0$. This finishes the proof.
\gdm

\section{Existence of a  solution to \equ{MP1}}
\label{S3}
\setcounter{equation}{0}
\setcounter{theorem}{0}
Recall the definition of the approximating process $\apmeX$ from \equ{E2.8} and the martingale problem \equ{MP1} from Theorem~\ref{T1.1}. In this section, we show that this process, in fact, converges to a solution of the martingale problem \equ{MP1} if $m\to\infty$, $\ve\downarrow0$. Since the order of limits does not play a role here, we assume that we are given a sequence $\ve_m\downarrow0$ and define $\apmX:=\apv{X}{m,\ve_m}$. Similarly, we define $\apm{I}:=\apv{I}{m,\ve_m}$, $\apm{J}$, $\apm{\MM}$ and so on.

Recall that $D_{\LBE}$ is the Skorohod space of c{\`a}dl{\`a}g functions $[0,\infty)\to\LBE$ equipped with the Skorohod topology.
This and the next section will be devoted to the proof of the following theorem.
\begin{theorem}
\label{T3.1}
Let  $\apmX_0=x\in\LBE$ for all $m\in\N$.
As $m\rightarrow \infty$,
the processes $\apmX$ converge in distribution
in $D_{\LBE}$ to $X$ which is the
unique solution to the martingale problem \equ{MP1} with $X_0=x$.
\end{theorem}
The strategy of the proof is pretty much standard. First we
prove tightness of the sequence of approximate processes and show that
every convergent subsequence satisfies the above martingale problem.
Then (in Section~\ref{S4}) we will show the uniqueness of the solution to the
martingale problem \equ{MP1}.

This section is devoted to the proof of the following proposition which is the first step in the
proof of the above theorem.
\begin{proposition}
\label{P3.2}
Let  $\apmX_0=x\in\LBE$ for all $m\in\N$.
\begin{enumerate}[(i)]
\item
The sequence $(\apmX)_{m\in\N}$ is tight
in $D_{\LBE}$.
\item
Any limit point of $(\apmX)_{m\in\N}$ in $D_{\LBE}$
solves  the martingale problem \equ{MP1}.
\end{enumerate}
\end{proposition}

\subsection{Proof of Proposition~\ref{P3.2}(i): Tightness}
The strategy for showing tightness in Proposition~\ref{P3.2} is to do two things:
\begin{enumerate}[(1)]
\item We show that the so-called compact containment condition holds
for $(\apmX)_{m\in\N}$ (see Lemma~\ref{L3.4}).
\item
Let $\mathrm{Lip_{\,f}}(\LBE;\C)$ denote the space of bounded Lipschitz functions on $\LBE$ that depend on only finitely many coordinates. We use moment estimates for the coordinate processes $\apmX(k)$ and Aldous's criterion to show that for $f\in\mathrm{Lip_{\,f}}(\LBE;\CC)$, the sequence $\big(f(\apmX_t)\big)_{t\geq0}$, $m\in\N$, is tight in $D_\R$ (Lemma~\ref{L3.5}).
\end{enumerate}
By the Stone-Weierstra{\ss} theorem, $\mathrm{Lip_{\,f}}(\LBE;\C)\subset C_b(\LBE;\C)$ is dense in the topology of uniform convergence on compacts. Hence (1) and (2) imply tightness of  $(\apmX)_{m\in\N}$  by Theorem~3.9.1 of \cite{bib:kur86}.

\subsubsection{Compact containment}

\begin{lemma}
\label{L3.3}
A subset $L\subset \LBE$ is relatively compact if and only if
\begin{enumerate}[(i)]
\item
$\sup_{y\in L}\|y_i\|_\beta<\infty$ for $i=1,2$, and
\item
for every $\delta>0$ there exists a finite $F\subset S$ such that
$$\sup_{y\in L}\big\|y_i\1_{S\setminus F}\big\|_\beta<\delta\mf i=1,2.$$
\end{enumerate}
\end{lemma}
\paragraph{Proof.}
Simple. (Note that in the setting of Prohorov's theorem, these two conditions correspond to boundedness of the total masses and to tightness.)
\gdm

\begin{lemma}[Compact containment condition]
\label{L3.4}
For every $T>0$ and every $\delta>0$, there exists a compact set $L_\delta\subset \LBE$ such that for every $m\in \N$
$$\P\big[\apmX_t\in L_\delta\mbs{for all}t\in[0,T]\big]\geq 1-\delta.
$$
\end{lemma}
\paragraph{Proof.}
Fix $T>0$. By Lemma~\ref{L3.3}, it is enough to show the following:
\begin{enumerate}[(i)]
\item
For any $\delta>0$, there exists $K>0$ such that for all $m\in\N$, we have
$$
\P\left[\sup_{t\leq T}\big\langle \apmX_{i,t},\beta\big\rangle >K\right]\leq
\delta,\mf i=1,2.
$$
\item
For any $\delta>0$, there exists a finite set $F\subset S$ such that for all $m\in\N$,
$$
\P\left[\sup_{t\leq T}\big\langle \apmX_{i,t} \1_{F^c},
\beta\big\rangle >\delta\right]\leq
\delta,\mf i=1,2.
$$
\end{enumerate}

While (i) is immediate from Lemma~\ref{L2.7}, for (ii), note that
\[\P\left[\sup_{t\leq T}\big\langle \apmX_{i,t} \1_{F^c},
\beta\big\rangle >\delta\right]
\leq \delta^{-1}\big\langle \bS_T\,x,\beta\1_{F^c}\big\rangle
\downarrow0\mbs{as}F^c\downarrow\emptyset.
\eope\]
\subsubsection{Tightness of coordinate processes}
The next step, which completes the proof of Proposition~\ref{P3.2}(i), is to show the following lemma.

\begin{lemma}
\label{L3.5}
Let $f\in\mathrm{Lip_{\,f}}(\LBE;\CC)$. Then $f\big(\apmX_t\big)_{t\geq0}$, $m\in \N$, is tight in the space $D_\CC$ of c{\`a}dl{\`a}g functions $[0,\infty)\to\CC$ equipped with the Skorohod topology.
\end{lemma}
\paragraph{Proof.}
Let $T>0$, $(\tau_m)_{m\in\N}$ and $r_1\in (0,1)$ be as in Lemma~\ref{L2.8}. Since $f$ is Lipschitz and depends on only finitely many coordinates, by Lemma~\ref{L2.8}, we have
$$
\lim_{\delta\downarrow 0}\limsup_{m\to\infty}\E\big[\big|f\big(\apmX_{\tau_m+\delta}\big)-f\big(\apmX_{\tau_m}\big)\big|^{r_1}\big]=0.
$$
Hence by Aldous's tightness criterion (see \cite{Aldous1978}), the claim follows.
\gdm

\subsection{Proof of Proposition~\ref{P3.2}(ii): Martingale problem for limit points}
In the previous subsection, we proved that the sequence of laws of $(\apmX)_{m\in \N}$ is tight in $D_{\LBE}$ and is hence relatively compact by Prohorov's theorem. Let $X$ be a process who's law is an arbitrary limit point of that sequence.
Then there exists  a subsequence $(\apv{X}{m_k})_{k\in\N}$
 such that
\[ \apv{X}{m_k}\;\wlimk\; X\]
weakly in $D_{\LBE}$.
In order to ease the notation, in this section we will assume that the sequences $(\ve_m)_{m\in\N}$ and $(S_m)_{m\in\N}$ were chosen such that
\[ \apmX\;\wlimm\; X.\]

\begin{remark}
\label{R3.6}
By \equ{E2.12} and \equ{E2.14} in Lemma~\ref{L2.4} and Remark~\ref{R2.5}, using Fatou's lemma, we get that \equ{E1.28} holds for the limiting process $X$ and hence the stochastic integral in \equ{E1.26} is well defined.
\end{remark}
First, we derive estimates on the first and second moment of a limit point $X$.
\begin{lemma}
\label{L3.7}
For all $t>0$, $k,l\in S$ with $k\neq l$ and $i=1,2$ we have
\begin{equation}
\label{E3.1}
\begin{aligned}
\E\left[ X_{i,t}(k)\right] &\,\leq\,  \bS_t X_{i,0}(k),\\[1mm]
\E\left[ X_{1,t}(k)X_{2,t}(l)\right] &\,\leq\,  \bS_t X_{1,0}(k)\,\bS_tX_{2,0}(l).
\end{aligned}
\end{equation}
For every $p\in(0,1]$, there exists a constant $c_p$ such that
\begin{equation}
\label{E3.2}
\E\Big[ \sup_{t\leq T} X_{i,t}(k)^p\Big]\leq c_p\bigg(1+x_i(k)^p+\int_0^T \big[ (\bA+1) \bS_sx_1(k)\big]
\big[(\bA+1)\bS_sx_2(k)+1\big]\,ds\bigg).
\end{equation}
Moreover for any  non-negative function $\phi$ on $S$, $T,K>0$, and $i=1,2$,
\begin{equation}
\label{E3.3}
\P\left[\sup_{t\leq T} \langle X_{i,t}\,,\phi\rangle >K\right]
\leq K^{-1}\big\langle \bS_Tx_i,\,\phi\big\rangle.
\end{equation}

\end{lemma}
\paragraph{Proof.}
The inequalities in \equ{E3.1} follow from Corollary~\ref{C2.3}
with the help of Fatou's lemma by
switching to the Skorohod space with the a.s. convergence instead of weak convergence of the processes.

For the same reasons, \equ{E3.2} follows from Lemma~\ref{L2.6}. Here we also used
the trivial inequality $a^p\leq a+1$ for $p\leq 1$ and $ a\geq 0$.

Equation \equ{E3.3} follows from Lemma~\ref{L2.7} again by the properties of the weak convergence.
\gdm

Now we have to identify the equation for the
limiting point $X$. For this goal, it will be enough to identify the compensator measures of the limits of the  martingales  $\apm{\MM}_i(k)$. At this stage  it will be more convenient for us to use a different representation of those processes. Let $\apm{\CN}_\Delta(\{k\},\ARG),$ $k\in S$, be the family of  point process on $ \R_+\times  \R\times \R$  induced by the jumps of the processes $\apmX$, that is
$$
\apm{\CN}_\Delta\big(\{k\},dt,dz\big)=\sum_{s}\1_{\{\Delta \apmX_s(k)\not=0\}}\delta_{(s,\,\Delta \apmX_s(k))}(dt,dz).
$$
Let $\apm{\CN}_\Delta{}^{\prime}$ denote the corresponding compensator measure and let $\apm{\CM}_\Delta:=\apm{\CN}_\Delta-\apm{\CN}_\Delta{}^{\prime}$. Furthermore, define $\CN_\Delta$, $\CN_\Delta'$ and $\CM_\Delta$ similarly, but with $\apmX$ replaced by $X$.

Recall from \equ{E2.11} that
$$\begin{aligned}
\apmX_{t,i}(k)=\,&
\apmX_{0,i}(k)
+\int_{0}^{t}\apmCA \apmX_{s,i}(k)\,ds\\
&+ \int_{0}^{t}\int_{[0,\infty)\times E^\ve} \apm{J}_i\left(y,\apmX_{s-}(k)\right)\1_{[0,\apm{I}(\apmX_{s-};k)]}(a)\,\CM\big(\{k\},ds,d(a,y)\big)\\
&+\int_0^t\int_{[0,\infty)\times E} \ve y_2\,\1_{[0, \apmIN_i(\apmX_{s-};k)]}(a)\CM^i\big(\{k\},ds,d(a,y)\big).
\end{aligned}
$$
Let $e_1:=(1,0)$ and $e_2:=(0,1)$.
As $\CM$, $\CM^1$ and $\CM^2$ are independent compensated jump measures with intensity $\CN'$ and since $\CN'$ is absolutely continuous (which implies that there are no double points), we get for $B\subset\R^2$ measurable that
$$\begin{aligned}\apm\CM_\Delta(\{k\},dt,B)=&\int_{[0,\infty)\times E^\ve} \1_{B\setminus\{0\}}\big(\apm J\big(y,\apmX_{t-}(k)\big)\big)\,\1_{[0,\apm{I} (\apmX_{s-};k)]}(a)\,\CM\big(\{k\},dt,d(a,y)\big)\\
&+\sum_{i=1}^2\int_{[0,\infty)\times E} \1_{B}(\ve_my_2 e_i)\,\1_{[0,I^{0,(m)}_i(\apmX_{s-};k)]}(a)\,\CM^i\big(\{k\},dt,d(a,y)\big).\end{aligned}
$$
Hence for $i=1,2$, we have
$$\begin{aligned}
\apmX_{t,i}(k)=
x_i(k)
+\int_{0}^{t}\apmCA \apmX_{s,i}(k)\,ds&+ \int_0^t \int_{\R^2}z_i\,\apm\CM_{\Delta}(\{k\},ds,dz).
\end{aligned}
$$
\begin{lemma}
\label{L3.8}
The weak limit point $X$ is a solution of
\begin{equation}
\label{E3.4}
X_{t}(k)\;=\;
x(k)
+\int_{0}^{t}\CA X_{s}(k)\,ds+ \int_0^t \int_{\R^2}z \,\CM_\Delta\big(\{k\},ds,dz\big),
\end{equation}
where
$$\CM_\Delta=\CN_\Delta-\CN_\Delta'.$$
The compensator measure of the point process $\CN_\Delta$ is given by
$$\CN_\Delta'\big(\{k\},dt,B\big)=\int_E \1_{B\setminus\{0\}} J(y, X_{t-}(k))\,I(X_t;k)\,dt\,\nu(dy)\mf B\subset \R^2\mbox{ measurable},\;\; k\in S.$$
\end{lemma}
\begin{proof}
By Theorem IX.2.4 of \cite{bib:jacshir87}, it is enough to check for all $k\in S$ that
\begin{equation}
\label{E3.5}
\Bigg(\apmX(k), \int_{\R^2}\apm{\CN}_\Delta{}^{\prime}\big(\{k\},dt,dz\big)\,G(z)\bigg)\stackrel{m\to\infty}{\Longrightarrow} \bigg(X(k), \int_{\R^2}\CN_\Delta^{\,\prime}\big(\{k\},dt,dz\big)\,G(z)\bigg),
\end{equation}
for
\begin{itemize}
\item[(i)]
each continuous $G\in C_b^+(\R^2)$ which is $0$ in some neighbourhood of $0$ and
\item[(ii)]
$G=h_ih_j$, $i,j=1,2$, for some bounded continuous function $h=(h_1,h_2):\R^2\to\R^2$ that fulfils $h(x)=x$ in some neighbourhood of $0$.
\end{itemize}
Note that
\begin{equation}
\label{E3.6}
\begin{aligned}
\int_{\R^2}G(z)\,\apm\CN_{\Delta}{}^\prime\big(\{k\},dt,dz\big)
=&
\int_{E^{\ve_m}}G\big(\apm J(y,\apmX_{t-}(k))\big)\,\apm\CN{}^{\prime}
\big(\{k\},dt, [0,\apm I(X_{t-};k)],dy\big)\\
&\,+
\sum_{i=1}^2\int_{E}G(\ve_my_2 e_i)
\,\apm{\CN}{}^{\prime}
\Big(\{k\},dt,[0,\apmIN_i(X_{t-};k)],dy\Big)\\
=&
\,\apm{I}(X_{t-};k)\bigg(\int_{E^{\ve_m}}G\big(\apm J(y,\apmX_{t-}(k))\big)\,\nu(dy)\bigg)\,dt\\
&\,+
\sum_{i=1}^2\apmIN_i(X_{t-};k)\bigg(\int_{E}G(\ve_my_2 e_i)\,\nu(dy)\bigg)\,dt\\
\end{aligned}
\end{equation}
and similarly
\begin{equation}
\label{E3.7}
\int_{\R^2}G(z)\,\CN^{\,\prime}_\Delta\big(\{k\},dt,dz\big)
\;=\;
I(X_{t-};k)\bigg(\int_{E}G\big(J(x,X_{t-}(k))\big)\,\nu(dx)\bigg)\,dt.
\end{equation}
By Skorohod's representation theorem, we may assume that $\apmX$ and $X$ are defined on one probability space such that $\apmX$ converges almost surely to $X$ (and not only weakly).
In order to show \equ{E3.5}, it is enough to show that the right hand side of \equ{E3.6} converges to that of \equ{E3.7}. In order to show this, it is enough to show that for any $k\in S$, uniformly in $z$ on compacts of $\LBE$, we have
\begin{equation}
\label{E3.8}
 \apm I(z;k) \int_{E^{\ve_m}}G\big(\apm J(y,z(k))\big)\,\nu(dy)\longrightarrow I(z;k)\int_{E}G\big(J(y,z(k))\big)\,\nu(dy)  \mas m\rightarrow\infty,
\end{equation}
and for $i=1,2$,
\begin{equation}
\label{E3.9}
\apmIN_i(z;k) \int_{E}G(\ve_m y_2 e_i)\,\nu(dy) \longrightarrow 0\mas m\rightarrow\infty.
\end{equation}
The proof of Lemma~\ref{L3.8} is thus complete when we have shown the following two lemmas.
\end{proof}
\begin{lemma}
\label{L3.9}
For any bounded measurable function $G:\R^2\to\R$ such that $G(x)=0$ for all $x\in\R^2$ with $\|x\|_\infty\leq\delta$ for some $\delta>0$, we have \equ{E3.8} and \equ{E3.9}.
\end{lemma}
\begin{lemma}
\label{L3.10}
For any $j,j'=1,2$, and for $G(x):=G_{jj'}(x):=x_jx_{j'}\1_{\{\|x\|_\infty\leq 1\}}$, we have \equ{E3.8} and \equ{E3.9}.
\end{lemma}

Fix $k\in S$ and a compact set $L\subset\LBE$. Then there exists a $K<\infty$ such that $|z_i(k)|\leq K$ and $|\apmCA z_i(k)|\leq K$ for all $z\in L$, $i=1,2$. \par

For $m$ large enough such that $S_m\ni k$, we have $\apmCA z(k)\uparrow \CA z(k)$. Since $\CA z(k)$ and $\apmCA z(k)$ are continuous functions of $z$ (see \equ{E1.8}), we get uniform convergence on compacts, that is,
\begin{equation}
\label{E3.10}
\eta_m:=\sup_{z\in L}\left|\apmCA z(k)-\CA z(k)\right|\limm0.
\end{equation}

\par\textbf{Proof (of Lemma~\ref{L3.9}).}\quad
\textbf{Case 1: $z(k)=0$.\quad}If $z(k)=0$, then both sides in \equ{E3.8} equal zero and it remains to show \equ{E3.9}.
Note that $\apmIN_i(z;k)\leq K/\ve_m$ and recall that $G(x)=0$ if $\|x\|\leq \delta$. Hence, by Lemma~\ref{LA.1},
$$
\left|\apmIN_i(z;k) \int_{E}G(\ve_m y_2 e_i)\,\nu(dy) \right|
\leq \frac{K\,\|G\|_\infty}{\ve_m}\,\nu\big(\{0\}\times(\delta/\ve_m,\infty)\big)
\leq \frac{2}{\pi}\frac{K\,\|G\|_\infty}{\delta^2}\,\ve_m.
$$
\par
\textbf{Case 2: $z(k)\neq0$.\quad}
In this case, the left hand side of \equ{E3.9} equals zero and it remains to show \equ{E3.8}.
We consider, without loss of generality, the case $z_1(k)>0$.
By Corollary~\ref{CA.6},
\begin{equation}
\label{E3.11}
\begin{aligned}
\left|\int_EG(J(y,z(k)))\nu(dy)\right|
&\leq \|G\|_\infty\,\nu\big(\big\{y:\,\|J(y,z(k))\|_\infty>\delta\big\}\big)\\
&= \|G\|_\infty\,\nu\big(\big\{y:\,\|J(y,(1,0))\|_\infty>\delta/z_1(k)\big\}\big)\\
&\leq \|G\|_\infty\,\frac4\pi\frac{z_1(k)^2}{\delta^2}\leq \|G\|_\infty\,\frac4\pi\frac{K^2}{\delta^2} .
\end{aligned}
\end{equation}

Assume that $\ve_m<\delta/K$.
\par
\textbf{Case 2(i): $\ve_m\leq z_1(k)$.\quad}In this case, by \equ{E3.10},
\begin{equation}
\label{E3.12}
\left|\apm I(z;k)-I(z;k)\right|=\left|z_1(k)^{-1}\apmCA z(k)-z_1(k)^{-1}\CA z(k)\right|\;\leq\; z_1(k)^{-1}\,\eta_m.
\end{equation}
For any $y\in E$, we have $J_i(y,z(k))=\apm J_i(y,z(k))$, $i=1,2$. For $y\in E\setminus E^{\ve_m}$, we have $J_2(y,z(k))=0$ and
$$|J_1(y,z(k))|\;\leq\; \ve_m z_1\;\leq\; \ve_m K\;\leq\; \delta.$$
Hence $G(J(y,z(k)))=0$ for $y\in E\setminus E^{\ve_m}$. This shows
\begin{equation}
\label{E3.13}
\int_{E^{\ve_m}}G\big(\apm J(y,z(k))\big)\,\nu(dy)= \int_{E^{\ve_m}}G\big(J(y,z(k))\big)\,\nu(dy)=\int_{E}G\big(J(y,z(k))\big)\,\nu(dy).
\end{equation}
Using  \equ{E3.13} and \equ{E3.11}, the difference of the left hand side and right hand side in \equ{E3.8} is bounded by
$$\|G\|_\infty\,\frac4\pi\frac{z_1(k)^2}{\delta^2}
\left|\apm I(z;k)-I(z;k)\right|\leq\frac{4K\,\|G\|_\infty}{\pi\delta^2}\,\eta_m\limm0,
$$
where the convergence follows by \equ{E3.10}.
\par
\textbf{Case 2(ii): $\ve_m> z_1(k)$.\quad}In this case, by \equ{E3.11},
$$\left|I(z;k)\int_{E}G\big(J(y,z(k))\big)\,\nu(dy)\right|\leq
\frac{K}{z_1(k)}\|G\|_\infty\,\frac4\pi\frac{z_1(k)^2}{\delta^2}\leq
\frac{4K\|G\|_\infty}{\pi\delta^2}\,\ve_m
$$
and similarly
$$\left|\apm I(z;k)\int_{E^{\ve_m}}G\big(\apm J(y,z(k))\big)\,\nu(dy)\right|\leq
\frac{4K\|G\|_\infty}{\pi\delta^2}\,\ve_m.
$$

This shows \equ{E3.8} and \equ{E3.9} and finishes the proof of Lemma~\ref{L3.9}.
\gdm\medskip
\par
\textbf{Proof (of Lemma~\ref{L3.10}).}\quad
The proof of this lemma uses some of the estimates from the proof of Lemma~\ref{L3.9}. Assume that $\ve_m\leq 1/3$.\par
\textbf{Case 1: $z(k)=0$.\quad}If $z(k)=0$, then both sides in \equ{E3.8} equal zero and it remains to show \equ{E3.9}.
Note that
$$\apmIN_i(z;k)\int_EG_{jj'}(\ve_my_2e_i)\,\nu(dy)=0\mfalls j\neq i\mbs{or}j'\neq i.$$
Hence, now assume that $j=j'=i$. Recall that $\apmIN_i(z;k)\leq K/\ve_m$. Then, by Lemma~\ref{LA.2},
$$\apmIN_i(z;k)\int_EG_{ii}(\ve_my_2e_i)\,\nu(dy)=
\apmIN_i(z;k)\,\ve_m^2\int_{\{0\}\times(0,1/\ve_m)}y_2^2\,\nu(dy)\leq \frac{4K}{\pi}\ve_m\log(1/\ve_m)\limm0.$$
That is, the right hand side in \equ{E3.9} equals $0$ while the left hand side converges to $0$.
\par
\textbf{Case 2: $z(k)\neq0$.\quad} In this case \equ{E3.9} holds trivially. Without loss of generality, we assume that $z_1(k)>0$.
\par
\textbf{Case 2a: $j=j'=1$.\quad}
We have
\begin{equation}
\label{E3.14}
\begin{aligned}
\int_{E^{\ve_m}}G_{11}\big(\apm{J}(y,z(k))\big)\,\nu(dy)
&\leq\int_{E^{\ve_m}}G_{11}\big(J(y,z(k))\big)\,\nu(dy)\\
&\leq\int_EG_{11}\big(J(y,z(k))\big)\,\nu(dy)\\
&=z_1(k)^2\int_E(y_1-1)^2\1_{\{|y_1-1|\leq 1/z_1(k)\}}\1_{\{y_2\leq1/z_1(k)\}}\,\nu(dy).
\end{aligned}
\end{equation}
For $z_1(k)>1$ this equals (using Lemma~\ref{LA.4} for the inequality)
$$
=z_1(k)^2\int_{(1-1/z_1(k),1+1/z_1(k))\times\{0\}}(y_1-1)^2\,\nu(dy)
\leq \frac{2}{\pi}z_1(k).
$$
For $z_1(k)\in(0,1]$, the right hand side in \equ{E3.14} equals (using Lemma~\ref{LA.1} and \ref{LA.4} for the inequality)
$$\begin{aligned}
z_1(k)^2\bigg[\int_{(0,1+1/z_1(k))\times\{0\}}(y_1-1)^2&\,\nu(dy)
+\nu\big(\{0\}\times(0,1/z_1(k))\big)\bigg]\\
&\leq z_1(k)^2\frac2\pi\big[2\log\big(1+1/z_1(k)\big)+1\big]\leq 2z_1(k).
\end{aligned}
$$
Summing up, for all $z_1(k)>0$, we have
\begin{equation}
\label{E3.15}
\int_{E^{\ve_m}}G_{11}\big(\apm{J}(y,z(k))\big)\,\nu(dy)
\leq\int_EG_{11}\big(J(y,z(k))\big)\,\nu(dy)\;\leq 2z_1(k).
\end{equation}
Now for the difference of the two integrals.
We have (using Lemma~\ref{LA.4})
\begin{equation}
\label{E3.16}
\int_{E\setminus E^{\ve_m}}G_{11}\big(J(y,z(k))\big)\,\nu(dy)
\leq z_1(k)^2\int_{(1-\ve_m,1+\ve_m)\times\{0\}}(y_1-1)^2\,\nu(dy)
\leq \frac2\pi\, z_1(k)^2\,\ve_m.
\end{equation}
Let
$$\Delta:=\int_{E^{\ve_m}}\big[G_{11}\big((J(y,z(k))\big)-G_{11}\big((\apm{J}(y,z(k))\big)\big]\,\nu(dy).$$
For $z_1(k)\geq\ve_m$, we have $\Delta=0$. On the other hand, for $z_1(k)\in (0,\ve_m)$, we have
(using Lemma~\ref{LA.1})
\begin{equation}
\label{E3.17}
\begin{aligned}
\Delta
&=
z_1(k)^2\int_{E^{\ve_m}}\big[\1_{\{|y_1-1|\leq1/z_1(k)\}}\1_{\{1/\ve_m<y_2\leq1/z_1(k)\}}\big](y_1-1)^2\,\nu(dy)\\
&=
z_1(k)^2\nu\big(\{0\}\times(1/\ve_m,1/z_1(k)]\big)\\
&\leq z_1(k)^2\,\nu\big(\{0\}\times(1/\ve_m,\infty)\big)\;\leq\; z_1(k)^2\,\ve_m^2.
\end{aligned}
\end{equation}
Putting the estimates \equ{E3.12}, \equ{E3.15}, \equ{E3.16} and \equ{E3.17} together and recalling that $\apm{I}(z;k)\leq I(z;k)\leq K/z_1(k)$, we get
$$
\begin{aligned}
\Bigg|I(z;k)\int_EG_{11}\big(J(y,z(k))\big)\,\nu(dy)
&-\apm{I}(z;k)\int_{E^{\ve_m}}G_{11}\big(\apm{J}(y,z(k))\big)\,\nu(dy)
\Bigg|\\
&\leq \big|I-\apm{I}\big|\left(\int_E+\int_{E^{\ve_m}}\right)
+(I+\apm{I})\left|\int_E-\int_{E^{\ve_m}}\right|\\
&\leq
\frac{\eta_m}{z_1(k)}\,4z_1(k)+\frac{2K}{z_1(k)}\,(\ve_m+\ve_m^2)\,z_1(k)^2\\
&\leq 4\eta_m+4K^2\ve_m\limm0.
\end{aligned}
$$
\par\textbf{Case 2b: $j=1$ and $j'=2$.\quad}
We have
$$\int_{E\setminus E^{\ve_m}}\big|G_{12}\big(J(y,z(k))\big)\big|\,\nu(dy)=0.$$
Since $J(y,z(k))=\apm{J}(y,z(k))$ and $I(z;k)=\apm{I}(z;k)$ if $z_1(k)\geq\ve_m$, in this case we infer
$$
I(z;k)\int_{E}G_{12}\big(J(y,z(k))\big)\,\nu(dy)
=\apm{I}(z;k)\int_{E^{\ve_m}}G_{12}\big(\apm{J}(y,z(k))\big)\,\nu(dy).
$$
Now assume $z_1(k)\in(0,\ve_m)$.
We have (using Lemma~\ref{LA.2})
\begin{equation}
\label{E3.18}
\begin{aligned}
I(z;k)\,\int_E\big|G_{12}\big(J(y,z(k))\big)\big|\,\nu(dy)
&=I(z;k)\,z_1(k)^2\int_E|y_1-1|\,y_2\,\1_{\{|y_1-1|\leq 1/z_1(k)\}}\,\1_{\{y_2\leq1/z_1(k)\}}\,\nu(dy)\\
&=I(z;k)\,z_1(k)^2\int_{\{0\}\times(0,1/z_1(k)]}y_2\,\nu(dy)\\
&\leq K\,z_1(k)\;\leq\; K\,\ve_m.
\end{aligned}
\end{equation}
Similarly,
\begin{equation}
\label{E3.19}
\begin{aligned}
\apm{I}(z;k)\,\int_{E^{\ve_m}}\big|G_{12}\big(\apm{J}(y,z(k))\big)\big|\,\nu(dy)
\;&=\; \apm{I}(z;k)\,z_1(k)\,\ve_m\int_{\{0\}\times(0,1/\ve_m]}y_2\,\nu(dy)\\
&\leq\; K\,z_1(k)\;\leq\; K\,\ve_m.
\end{aligned}
\end{equation}
\par\textbf{Case 2c: $j=j'=2$.\quad}
For $z_1(k)\geq\ve_m$, we have
$$\int_EG_{22}\big(J(y,z_1(k))\big)\,\nu(dy)
=\int_{E^{\ve_m}}G_{22}\big(\apm{J}(y,z_1(k))\big)\,\nu(dy)$$
and $I(z;k)=\apm{I}(z;k)$.
That is, \equ{E3.8} holds trivially. Hence now assume that $z_1(k)\in(0,\ve_m)$.
Then, by Lemma~\ref{LA.2} (noting that $z\mapsto z\log(1/z)$ is monotone increasing on $z\leq\ve_m\leq1/3$),
\begin{equation}
\label{E3.20}
\begin{aligned}
I(z;k)\int_EG_{22}\big(J(y,z_1(k))\big)\,\nu(dy)
&=I(z;k)\,z_1(k)^2\int_{\{0\}\times(0,1/z_1(k)]}y_2^2\,\nu(dy)\\
&\leq \frac{4K}\pi\,z_1(k)\log\big(1/z_1(k)\big)\\
&\leq\frac{4K}{\pi}\,\ve_m\,\log\big(1/\ve_m\big).
\end{aligned}
\end{equation}
Similarly,
\begin{equation}
\label{E3.21}\begin{aligned}
\apm{I}(z;k)\int_{E^{\ve_m}}G_{22}\big(\apm{J}(y,z_1(k))\big)\,\nu(dy)
&=\apm{I}(z;k)\,\ve_m^2\,\int_{\{0\}\times(0,1/\ve_m]}y_2^2\,\nu(dy)\\
&\leq \frac{4K}\pi\,\ve_m\,\log\big(1/\ve_m\big).
\end{aligned}
\end{equation}
The expressions in \equ{E3.20} and \equ{E3.21} (and hence their differences) are bounded by $(4K/\pi)\ve_m\log(1/\ve_m)\limm0$.
\gdm
\medskip

For $y\in\R^2$, define $h_y:E\to\CC$ by
$$
h_y(z):=e^{(z-(1,0))\mtimes y}-1-(z-(1,0))\mtimes y,
$$
Furthermore, for $x\in E$, let
\begin{equation}
\label{E3.22}
h_{x,y}(z):=\cases{h_{x_1y}(z),&\mfalls x_1>0,\\[1mm]
h_{x_2(y_2,y_1)}(z),&\mfalls x_2>0,\\[1mm]
0,&\msonst.}
\end{equation}
Note that
$$
h_{x,y}(z)=e^{J(z,x)\mtimes y}-1-J(z,x)\mtimes y.
$$
\begin{lemma}
\label{L3.11}
For any $x,y\in E$, we have
$$\int_E h_{x,y}\,d\nu=0.$$
\end{lemma}
\textbf{Proof.}
By symmetry, it is enough to consider the case $x=(1,0)$. Note that $h_{(1,0),y}=h_y$.

Let $B$ be planar Brownian motion started at $z\in[0,\infty)^2$ and recall that $Q_z$ is its harmonic measure on $[0,\infty)^2$; that is, $Q_z$ is the distribution of $B_\tau$ where $\tau$ is the exit time from $(0,\infty)^2$.

By formally extending the domain of $h_y$ to $[0,\infty)^2$, and applying It{\^o}'s formula, we see that $(h_y(B_t))_{t\geq0}$ is a ($\CC$-valued) martingale. Since $h_y$ grows at most linearly, we have $\E[|h_y(B_t)|^2]\leq C\,\E[\|B_t\|^2]\leq C(\|z\|^2+2t)$ for some $C<\infty$. It is well known that $\E[\tau^p]<\infty$  for $p\in[1/2,1)$ (see \cite[Lemma 3.5]{KM1} or \cite[Equation (3.8)]{Burkholder1977} with $\alpha=\pi/2$)). Hence, by Burkholder's inequality, $|h_y(B_{t\wedge \tau})|$, $t\geq0$, is bounded in $L^p$ for all $p\in[1,2)$. Applying the martingale convergence theorem, we get $\int h_y \,dQ_z=h_y(z)$.

Recall that $\nu$ is the vague limit of $\ve^{-1}Q_{(1,\ve)}$ as $\ve\to0$. Hence we can hope that $\int h_y\,d\nu$ can be written as the limit of $\ve^{-1}\int h_y\,dQ_{(1,\ve)}$. In fact, since $h_y$ grows at most linearly and since $h_y(1,0)=0$, by \cite[Lemma 5.5]{KM1}, we get
\[
\int h_y\,d\nu
=\lim_{\ve\to0}\frac{1}{\ve}
\int h_y\,dQ_{(1,\ve)}\\
=\lim_{\ve\to0}\frac{1}{\ve}
h_y(1,\ve)=0.
\eope\]

Now we are ready to write the martingale problem for any limiting point
$X$.

\begin{lemma}
\label{L3.12}
Let $y\in \LFE$ and let $X$ be any limit point of $(\apmX)_{m\in\N}$. Then
\begin{equation}
\label{E3.23}
M_t:=e^{\lb X_t,y\rb}-e^{\lb X_0,y\rb}-\int_0^t\lb\CA X_s,y\rb\,e^{\lb X_{s},y\rb}\,ds,
\end{equation}
is a martingale.
\end{lemma}

\begin{proof}
By It{\^o}'s formula for discontinuous semimartingales (see, e.g., \cite[Theorem 32]{bib:pr04}) applied to $X$  solving \equ{E3.4}, we get that
\begin{equation}
\label{E3.24}
\begin{aligned}
e^{\lb X_t,y\rb}&-e^{\lb X_0,y\rb}-\int_0^t\lb\CA X_s,y\rb\,e^{\lb X_{s},y\rb}\,ds\\
&\quad-\sum_{k\in S}\int_0^t\int_E\CN'\big(\{k\},ds,dz\big)\,e^{\lb X_{s},y\rb}\big[e^{J(z, X_{s}(k)) y(k)}-1-J(z, X_{s}(k))\mtimes y(k)\big]
\end{aligned}
\end{equation}
is a local martingale.
By the definition of $\CN'$ and $h_{x,y}$ in \equ{E3.22}, using Lemma~\ref{L3.11}, we get
$$
\begin{aligned}
\int_0^t\int_E
\CN'(\{k\}&,ds,dz)e^{\lb X_{s},y\rb}\big[e^{J(z, X_{s}(k))\mtimes y(k)}-1-J(z, X_{s}(k))\mtimes y(k)\big]\\
=&\int_0^te^{\lb X_{s},y\rb}\int_Eh_{X_{s}(k),y(k)}\,d\nu\;=\;0.\end{aligned}
$$
Hence also $M$ is a local martingale and it remains to show that $M$ is in fact a martingale.
Applying \equ{E3.2}, \equ{E1.10}, \equ{E1.9},
for all $T>0$, we get
$$
\begin{aligned}
\E\bigg[\sup_{t\leq T}\bigg|\int_0^t\lb\CA X_s,y\rb\,e^{\lb X_{s},y\rb}\,ds\bigg|\bigg]
&\leq \E\left[\int_0^T
\big\langle \bA X_{1,s}\,, |y_1|\big\rangle +\big\langle \bA X_{2,s}, |y_2|\big\rangle\,ds\right]\\
&\leq \int_0^T
\big\langle \bA \bS_sx_1+\bA \bS_sx_2\,, |y|\big\rangle\,ds\\
&\leq \sum_{k\in S} \sum_{i=1}^2 \frac{e^{\LSC s}\left\Vert x_i\right\Vert_{\beta}}{\beta(k)}|y(k)|
\;<\;\infty.
\end{aligned}
$$
Note that the last inequality follows since $y$ has finite support.
Since the exponents in \equ{E3.23} have nonpositive real part, they are bounded. Hence we conclude that
\[ \E\left[\sup_{t\leq T} |M_t|\right]<\infty.\]
But this implies  that $M$ is indeed a martingale.
\end{proof}

By Lemma~\ref{L3.12}, any limit point of $(\apmX)_{m\in\N}$ solves the martingale problem \equ{MP1}. Hence the proof of Proposition~\ref{P3.2}(ii) is now complete.
\gdm

\section{Uniqueness of solutions to the martingale problem \equ{MP1}}
\label{S4}
\setcounter{equation}{0}
\setcounter{theorem}{0}
This section is devoted to the proof of the following proposition.
\begin{proposition}
\label{P4.1}
There is a unique solution to the martingale problem~\equ{MP1} and the map $x\mapsto P_x$ is measurable.
\end{proposition}
The proposition will be proved via a series of lemmas.
\subsection{The dual martingale problem}
\label{S4.1}
Recall from \equ{E1.14} that $\LFE$ is the space of $y\in E^S$ with only finitely many nonzero coordinates.
Recall that $\CA^*$ and $\bA^*$ are the transpose matrices of $\CA$ and $\bA$, respectively, and that $\cS^*$ and $\bS^*$ are the corresponding semigroups.
Recall the definition of $\LIBE$ from \equ{E1.15} and let $D_{\LIBE}=D_{\LIBE}[0,\infty)$ be the Skorohod space of $\LIBE$-valued c{\`a}dl{\`a}g paths.

We will define a $D_{\LIBE}$ valued process $Y=(Y_1,Y_2)$ that solves the martingale problem  which is \emph{dual} to \equ{MP1}. Recall the function $H$ from \equ{E1.12}.
\begin{proposition}
 \label{P4.2}
Let $Y_0=y\in \LFE$.
Then there exists the process $Y\in D_{\LIBE}$ which satisfies  the following martingale problem: For all $x\in\LBE$,
\[
\label{MP1*}
M^{*,x,y}_t:=H(x,Y_t) -H(x,Y_0)-\int_0^t \Lb x,\CA^*Y_s\Rb\, H(x,Y_s)\,ds\tag{MP$^*_1$}
\]
is martingale.
\end{proposition}
\paragraph{Proof.} The existence of a process $Y\in D_{\LBE}$ that solves the
martingale problem \equ{MP1*} for all $x\in\LFE$ follows immediately from Proposition~\ref{P3.2}, since the
assumptions on $\CA$ are satisfied by $\CA^*$ as well. In fact, by assuming that $Y$ is constructed similarly as $X$ in Section~\ref{S3}, we may assume that Lemma~\ref{L3.7} holds for $Y$. To finish the proof  we have to show  that this $Y$ takes in fact values in the subspace $\LIBE$ and that
$Y$ satisfies \equ{MP1*} for all $x\in\LBE$ (not only for $x\in\LFE$).

\textbf{Step 1.}\quad First we show that  $Y$ takes values in $\LIBE$. It is enough to show that
for all $\phi\in\LB$ and $i=1,2$, we have
\begin{equation}
\label{E4.1}
\P\left[\sup_{t\leq T} \big\langle Y_{i,t}\,,\phi\big\rangle >K\right]
\to 0\mas K\to \infty.
\end{equation}
By \equ{E3.3} in Lemma~\ref{L3.7}, for any $\phi\in \LB$ and $K>0$, we get (recall that $(\bS_T^*)$ is the semigroup generated by the transposed matrix $\bA^*$)
\begin{equation}
\label{E4.2}
\P\left[\sup_{t\leq T} \big\langle Y_{i,t}\,,\phi\big\rangle >K\right]
\leq K^{-1}\big\langle  \bS_T^* Y_{i,0},\;\phi\big\rangle=K^{-1}\big\langle  Y_{i,0},\;\bS_T\phi\big\rangle.
\end{equation}
By \equ{E1.11}, we have that $\bS_T\phi(k)<\infty$ for all $k$ and since $Y_{i,0}$ has finite support, the right hand side of \equ{E4.2} is finite.

\textbf{Step 2.}\quad
Now we show that $Y$ satisfies \equ{MP1*} for all $x\in\LBE$. Let
$(x_n)_{n\in\N}$ be a sequence in $\LFE$ such that
$$x_n\uparrow x\mas n\to\infty.$$
Then $M^{*,x_n,y}$ is a martingale for any $n\in\N$.
By \equ{E4.2}, for any $T>0$,
\begin{equation}
\label{E4.3}
\sup_{s\leq T}\big|H(x_n,Y_s)-H(x,Y_s)\big|\limn0\mfs\mbsl{(and hence in $L^1$).}
\end{equation}
Note that
$$|\lb x_n,\CA^*Y_s\rb|\;\leq\;
2\,\big\langle x_1+x_2,\,\bA^*(Y_{1,s}+Y_{2,s})\big\rangle
\;=\;2\,\big\langle \bA(x_1+x_2),Y_{1,s}+Y_{2,s}\big\rangle.
$$
Consequently, for all $T>0$ and $t\in[0,T]$, we get
\begin{equation}
\label{E4.4}
\begin{aligned}
\left|\int_0^t \lb x_n,\CA^*Y_s\rb\, H(x_n,Y_s)\,ds\right|
&\leq
\int_0^t |\lb x_n,\CA^*Y_s\rb|\,\,ds\\
&\leq
\int_0^T |\lb x_n,\CA^*Y_s\rb|\,\,ds\\
&\leq2\int_0^T\big\langle \bA(x_1+x_2),Y_{1,s}+Y_{2,s}\big\rangle\,ds.
\end{aligned}
\end{equation}
By Lemma~\ref{L3.7}, the expectation of the right hand side of \equ{E4.4} is bounded by
\begin{equation}
\label{E4.5}
\begin{aligned}
2\int_0^T\big\langle &\bA(x_1+x_2),\textbf{S}_s^*(Y_{1,0}+Y_{2,0})\big\rangle\,ds\\
&=
2\int_0^T\big\langle \textbf{S}_s\bA(x_1+x_2),Y_{1,0}+Y_{2,0}\big\rangle\,ds\\
&\leq
2e^{\LSC T}\big(\|x_1\|_\beta+\|x_2\|_\beta\big)\sum_{k\in S}\frac{Y_{1,0}(k)+Y_{2,0}(k)}{\beta(k)}<\infty.
\end{aligned}
\end{equation}
By dominated convergence, the integral term in the definition of $M^{*,x_n,y}$ converges in $L^1$ to the corresponding integral term for $M^{*,x,y}$. Hence $M^{*,x_n,y}_t$ converges in $L^1$ to $M^{*,x,y}$ for each $t$. Consequently, $M^{*,x,y}$ is a martingale.\gdm

\subsection{Moment bounds for solutions of the martingale problem}
\label{S4.2}
In Lemma~\ref{L3.7}, we established a bound on the first moments of those solutions $X$ of the martingale problem \equ{MP1} that arise as limiting points of the approximating processes $(\apmX)_{m\in\N}$. In order to show uniqueness of the solution to \equ{MP1},
we need to establish a similar bound for \emph{any} solution to \equ{MP1}. In fact we will establish a slightly stronger
result, but first let us define the notion of the {\it local} martingale problem. We say that $X$ solves {\it local}
 martingale problem \equ{MP1} with $X_0=x\in\LBE$ if for any $y\in\LFE$, the process $M^{x,y}$ is a local martingale.
Now we are ready to prove the following lemma.
\begin{lemma}
\label{L4.3}
Let $x\in \LBE$ and let $X$ be a solution to the local martingale problem \equ{MP1} with $X_0=x$. Then
\begin{itemize}
\item[(i)]
for all $k\in S$, $t\geq0$ and for $i=1,2$, we have
$$
\E\left[X_{i,t}(k)\right] \leq \bS_t (x_1+x_2)(k),$$
\item[(ii)] and $X$ is a solution to the martingale problem \equ{MP1} with
$X_0=x\in\LBE$.
\end{itemize}
\end{lemma}
\begin{proof}{\bf (i)}\quad
 Let $y=(1,1)\1_{\{k\}}\in\LFE$ be the test function that takes the value $(1,1)\in E$ at $k$ and is zero otherwise.
For $K>0$ define the stopping time
$$\tau_K = \inf\big\{t\geq0:\; \| X_{1,t}+X_{2,t}\|_{\beta}\geq K\big\}.
$$
Since a bounded local martingale is a martingale, and since for every $\ve>0$,
\begin{equation}
\label{E4.6}
e^{\lb X_{t\wedge\tau_K},\ve y\rb}
 - e^{\lb x,\ve y\rb}
 - \ve\int_0^{{t\wedge\tau_K}} e^{\lb X_s,\ve y\rb}
\lb \CA X_s,  y\rb\,ds
\end{equation}
is bounded by $2+2\ve T\Gamma K/\beta(k)$ for $t\leq T$, the expression in \equ{E4.6} is in fact martingale.
Hence, we have
$$
\E\Big[\ve^{-1}\big(1- e^{\lb X_{t\wedge\tau_K},\ve y\rb}\big)\Big]=
\ve^{-1}\big(1- e^{\lb x,\ve y\rb}\big)
-\E\left[\int_0^{{t\wedge\tau_K}} e^{\lb X_s,\ve y\rb}
\lb \CA X_s,  y\rb\,ds\right].
$$
Note that $\Re\lb x,\ve y\rb\leq0$ for all $x\in E^S$. Hence
$$
\Re\big(1- e^{\lb X_{t\wedge\tau_K}, \ve y\rb}\big)\geq0.
$$
Using Fatou's lemma we get
$$\begin{aligned}
\label{E}
\E\big[ X_{1,t\wedge\tau_K}(k)+&\,X_{2,t\wedge\tau_K}(k)\big]
\\
&=
\frac12\E\left[ \lim_{\ep \downarrow 0}\Re\ve^{-1}\big(1- e^{\lb X_{t\wedge\tau_K},\ve y\rb}\big)\right]\\
&\leq\frac12 \liminf_{\ep\downarrow 0} \
\ve^{-1}\Re\E\left[1- e^{\lb X_{t\wedge\tau_K},\ve y\rb}\right]\\
&=
 x_1(k)+x_2(k)
-\frac12\limsup_{\ep\downarrow 0} \Re\left(\E\left[\int_0^{{t\wedge\tau_K}} e^{\lb X_s, \ve y\rb}
\lb \CA X_s\,, y\rb\,ds\right]\right).
\end{aligned}
$$
Using dominated convergence (recall $\tau_K$), we obtain (recall \equ{E1.10} and $\LSC$ from \equ{E1.6})
\begin{equation}
\label{E4.7}
\begin{aligned}
\E\big[ X_{1,t\wedge\tau_K}(k)+&\,X_{2,t\wedge\tau_K}(k)\big]\\
&\leq  x_1(k)+x_2(k)
- \frac12\Re\left(\E\left[\int_0^{t\wedge\tau_K}
\lb \CA X_s, y\rb\,ds\right]\right)
\\
&\leq  x_1(k)+x_2(k) + \E\left[\int_0^{{t\wedge\tau_K}}
\big(\bA X_{1,s}(k)+\bA X_{2,s}(k)\big)\,ds\right]\\
&\leq  x_1(k)+x_2(k) +2K\LSC t/\beta(k)\;<\;\infty.
\end{aligned}
\end{equation}
From \equ{E4.7}, we get
$$
\begin{aligned}
\E\big[ X_{1,t\wedge\tau_K}&(k)+X_{2,t\wedge\tau_K}(k)\big]\\
&\leq
x_1(k)+x_2(k) +  \E\left[\int_0^{t}
\bA (X_{1,s\wedge\tau_K}+X_{2,s\wedge\tau_K})(k)\,ds\right].
\end{aligned}
$$
Since both side are finite by \equ{E4.7}, standard arguments yield
$$
\E\big[X_{1,t\wedge\tau_K}(k)+X_{2,t\wedge\tau_K}(k)\big]\leq
\bS_t (x_1+x_2)(k)\mfa t\geq 0\mbs{and} K\geq0.
$$
Letting $K\rightarrow \infty$ and using Fatou's lemma, we obtain
$$
\E\big[ X_{1,t}(k)+X_{2,t}(k)\big]\leq
\bS_t (x_1+x_2)(k).
$$
This finishes the proof of {(i)}.\medskip

{\bf (ii)}\quad
We have to show that the local martingale
\[ M^{x,y}_t=H(X_t,y) -H(x,y)-\int_0^t \lb\CA X_s,y\rb\, H(X_s,y)\,ds\]
is in fact a martingale. The argument is similar as in the proof of
Lemma~\ref{L3.12}. Therefore, we omit the details.
\end{proof}

\begin{corollary}
\label{C4.4}
Let $x\in \LBE$ and $X$ be a solution to the martingale problem \equ{MP1} with $X_0=x$ and let $\phi\in \LIB$. Then for all $t\geq0$ and $i=1,2$, we have
\begin{equation}
\label{E4.8}
\E\big[\langle X_{i,t},\phi\rangle\big]\, \leq \,\big\langle \bS_t (x_1+x_2),\phi\big\rangle\,\leq\, e^{\LSC t}\big\langle x_1+x_2,\phi\big\rangle\,<\,\infty.
\end{equation}
\end{corollary}
\begin{proof}
The first inequality is a consequence of the previous lemma, the second is due to \equ{E1.9} and the third is due to the very definition of $\LIB$.
\end{proof}
\begin{corollary}
\label{C4.5}
Let $X_0=x\in \LBE$ and $X$ be a solution to~\equ{E1.26}. Then $X$ is a solution to
the martingale problem \equ{MP1} with $X_0=x$.
\end{corollary}
\begin{proof}
By It\^o's formula  (see \equ{E3.24} and \equ{E3.23} in the proof of Lemma~\ref{L3.12})
we get that $X$ is a solution to the local martingale problem \equ{MP1}. Then by Lemma~\ref{L4.3}(ii), we get that
it is also a solution to the martingale problem \equ{MP1}.
\phantom{.....}\end{proof}

By definition, for any $x\in\LBE$ and any solution $X$ of the martingale problem \equ{MP1} with $X_0=x$, the process $M^{x,y}$ is a martingale for any $y\in\LFE$. The $L^1$-estimates we have just established enable us to show that this is true even for $y\in\LIBE$.
\begin{lemma}
\label{L4.6}
For any $x\in\LBE$, any solution $X$ of \equ{MP1} with $X_0=x$ and any $y\in\LIBE$, the process $M^{x,y}$ is a martingale.
\end{lemma}
\begin{proof}
The proof is similar to Step 2 of Proposition~\ref{P4.2}. For the key estimate of \equ{E4.5}, here we employ Corollary~\ref{C4.4} instead of Lemma~\ref{L3.7}. We omit the details.
\end{proof}
\subsection{Uniqueness via duality}
\label{S4.3}
\begin{proposition}[Duality]
 \label{P4.7}
Let  $Y_0=y\in \LFE$ and let $Y\in D_{\LIBE}$ be a solution to the martingale problem \equ{MP1*}.
Let $X_0=x\in \LBE$ and let
$X\in D_{\LBE}$ be an a solution to the martingale  problem
\equ{MP1} which is independent of $Y$. Then $X$ and $Y$ are dual with respect to the function $H$:
\begin{equation}
\label{E4.9}
\E\big[ H(X_t,Y_0)\big]=\E\big[ H(X_0,Y_t)\big]\mfa t\geq 0.
\end{equation}
\end{proposition}
\begin{proof}
Fix $t>0$. For $r,s\in[0,t]$, define
$$
\begin{aligned}
f(s,r)\;&=\; \E\big[ H(X_s,Y_r)\big]\mbsl{and} \\
g(s,r)\;&=\; \E\big[ \lb \CA X_s,Y_r\rb H(X_s,Y_r)\big]\;=\;\E\big[ \lb  X_s,\CA^*Y_r\rb H(X_s,Y_r)\big].
\end{aligned}
$$
By \equ{E4.5} and Corollary~\ref{C4.4}, we get
$$\begin{aligned}
\E\big[|M^{*,X_s,y}_r|\big]\;&\leq\; 2+2 e^{\LSC r}\,\E\big[\|X_{1,s}+X_{2,s}\|_\beta\big]\sum_{k\in S}\frac{y_1(k)+y_2(k)}{\beta(k)}\\
&\leq\; 2+4\,e^{\LSC(r+s)}\|x_1+x_2\|_\beta\sum_{k\in S}\frac{y_1(k)+y_2(k)}{\beta(k)}\;<\;\infty.\end{aligned}
$$
Hence we can compute
\begin{equation}
\label{E4.10}
\begin{aligned}
f(s,r)- f(s,0)- \int_0^r g(s,u)\,du
&\;=\;\E\big[M^{*,X_s,y}_r\big]\\
&\;=\;\E\big[\E[M^{*,X_s,y}_r\,\big|\,X_s]\big]\;=\;0,
\end{aligned}
\end{equation}
since $M^{*,X_s,y}$ is a martingale with $M^{*,X_s,y}_0=0$.
Similarly, we get
\begin{equation}
\label{E4.11}
f(s,r)-f(0,r)-\int_0^s g(u,r)\,du \;=\;\E\big[M^{x,Y_r}_s\big]\;=\;0.
\end{equation}
Using the same estimates for $\E\big[\lb \CA X_s,Y_r\rb\big]$, we obtain
\begin{equation}
\label{E4.12}
\int_0^t\int_0^t|g(r,s)|\,dr\,ds\;<\;\infty.
\end{equation}

By \equ{E4.10}, \equ{E4.11}, \equ{E4.12} and Lemma~4.4.10 of~\cite{bib:kur86} (with their $f_1$ and $f_2$ both equal to our $g$), we get $f(0,t)=f(t,0)$.
\end{proof}
\paragraph{Proof of~Proposition~\ref{P4.1}.}
\textbf{Step 1 (One-dimensional distributions).}\quad
Let $x\in\LBE$ and let $X, X'\in D_{\LBE}$ be two solutions to the
martingale problem \equ{MP1} with $X_0=X_0'=x$. Let $y\in\LFE$ and let $Y$ be a solution to \equ{MP1*} with $Y_0=y$. By Proposition~\ref{P4.7}, we have
\begin{equation}
\label{E4.13}
\E\big[ H(X_t,y)\big]=\E\big[ H(X_0,Y_t)\big]
=
\E\big[ H(X'_t,y)\big]
\mfa t\geq 0.
\end{equation}
By Corollary 2.4 of \cite{KM1}, the family $\{H(\ARG,y),\,y\in\LFE\}$ is measure determining, hence the one-dimensional marginals of $X$ and $X'$ coincide.\medskip

\textbf{Step 2 (Finite-dimensional distributions).}\quad
Now we use a version of the well-known theorem claiming that ``uniqueness of one-dimensional distributions for solutions to a martingale problem implies uniqueness of finite-dimensional distributions''. More precisely,
denote by $\cF_t=\sigma(X_s,\,s\leq t)$ the $\sigma$-algebra generated by $X_s$, $s\leq t$. Note that $(\LFE,\,\|\ARG\|_\beta)$ is a separable Banach space. Hence there exists a regular conditional probability
$Q_s=\P[(X_{s+t})_{t\geq 0}\in \ARG\,\big|\,\cF_s]$. Arguing as in \cite[Corollary VI.2.2]{bib:b97}, we see that for almost all $\omega$, under $Q_s$ the canonical process is a solution to \equ{MP1} started in $X_s$.

Now we may argue as in the proof of  Theorem~VI.3.2 in~\cite{bib:b97} to get uniqueness distribution of $X$.

\textbf{Step 3 (Measurability).}\quad
For the proof of the existence of a solution to \equ{MP1}, we employed an  approximation procedure: We constructed processes $\apmX$ with finitely many jumps (in finite time intervals) from a given noise, and showed convergence along a subsequence $m_n\uparrow\infty$. Due to uniqueness of the limit point (Step 2), we get convergence as $m\to\infty$. Let us denote the corresponding laws (with initial point $x$) by  $P^{m}_x$ and $P_x$.
By the very construction of $\apmX$ it is clear that $x\mapsto P^{m}_x$ is measurable. Hence also the limit $x\mapsto P_x$ is measurable.
\gdm

\paragraph{Proof of Theorems~\ref{T1.1} and \ref{T3.1}.}
Theorems~\ref{T1.1}(a) and \ref{T3.1} follow immediately from  Propositions~\ref{P3.2} and \ref{P4.1}. Theorem~\ref{T1.1}(b) follows from Lemma~\ref{L4.6}.

In order to show the strong Markov property of Theorem~\ref{T1.1}(c), by \cite[Theorem~4.4.2]{bib:kur86}, it is enough to show that the martingale problem \equ{MP1} is well-posed not only for deterministic points $x\in\LBE$, but also for probability measures $\mu\in\CM_1(\LBE)$. The problem is, of course, that for $X_0\sim \mu$ and $y\in\LFE$, in general the process $M^{X_0,y}$ is not well defined, as the integrand $\lb\CA X_s,y\rb H(X_s,y)$ is unbounded.
Hence, we propose a slight modification of \equ{MP1} and assume that $y\in\LIBEPP$, where
$$
\begin{aligned}
\LIBEPP:=\big\{y\in&\LBE:\,\exists c<\infty\mbs{with}\\
& c^{-1}\beta(k)<y_i(k)<c\beta(k)\,\forall\,i=1,2,\,k\in S\big\}
\subset\LIBE.
\end{aligned}
$$
Recall that $\|\CA u\|_\beta\leq \LSC\|u\|_\beta$ for all $u\in([0,\infty)^2)^S$. Hence for all $y\in\LIBEPP$, the map $\LBE\to\C$, $x\mapsto \lb\CA x,y\rb H(x,y)$ is bounded.
Hence for $y\in\LIBEPP$, the process $M^{X_0,y}$ is well defined, and we say
that $X_0$ as a solution to the martingale problem (MP$'$) if $M^{X_0,y}$ is a martingale for all $y\in\LIBEPP$. Arguing as in the proof of Proposition~\ref{P4.7}, we get the duality
\begin{equation}
\label{E4.14}
\E[H(X_t,y)]=\E[H(X_0,Y_t)]\mfa y\in\LIBEPP.
\end{equation}
Note that $\LIBEPP\subset\LIBE$ is dense. Hence \equ{E4.14} determines the distribution of $X_t$. By \cite[Theorem 4.4.2(a)]{bib:kur86}, we infer uniqueness of the finite-dimensional distributions and hence of the solution to (MP$'$). Hence $P_\mu:=\int \mu(dx)\,P_x$ is the unique distribution of any solution to (MP$'$) with $X_0\sim\mu$. That is, the martingale (MP$'$) is well-posed and hence by \cite[Theorem~4.4.2]{bib:kur86}, $(P_x)_{x\in\LBE}$ possesses the strong Markov property.\gdm

\section{Proof of Theorem~\ref{T1.3}}
\label{S5}
\setcounter{equation}{0}
\setcounter{theorem}{0}
First we will show weak uniqueness of the solutions of \equ{E1.26}. Let $X$ be any solution to~\equ{E1.26}
with $X_0=x\in\LBE$. Then by Corollary~\ref{C4.5}, we get that $X$ is also a solution to the martingale problem
\equ{MP1}. However, by Theorem~\ref{T1.1}, the solution to \equ{MP1} is unique in law. Hence also the solution to
\equ{E1.26} is unique in law.

Now we will show the existence of $(X,\CN)$ solving~\equ{E1.26}.
The procedure is pretty much standard and we only sketch the main arguments.

Let $X$ be the unique (in law)  solution to the martingale problem \equ{MP1}.
By Lemma~\ref{L3.8} and Theorem~\ref{T3.1}, we get that
$X$ can be constructed in a way that it also satisfies \equ{E3.4}.
Moreover, we define the point process $\tilde\CN$ by
$$\CN_\Delta(\{k\},dt,A)=\int_{[0,\infty)\times E]} \1_{A\setminus\{0\}}\big( J(y, X_{t-}(k))\big)\,\tilde\CN(\{k\},dt,dy),\mf A\subset \R^2.$$

Let $(k_n, t_n,  x_n)_{n\geq 1}$ be an arbitrary labeling of the points of the point process
$\tilde\CN$. Let $\CN^1$ be a Poisson point process on $S\times\R_+\times\R_+\times E$
independent of $\tilde\CN$ and $X$. Also let  $\{U_n\}_{n\geq 1}$ be a sequence of independent random variables
uniform on $(0,1)$ which are also independent of  $\tilde\CN$ and $X$.

 Define the new point process $\CN$ on $S\times\R_+\times\R_+\times E$ by
\begin{equation}
\label{E5.1}
\begin{aligned}
\CN(dk,dt,dr,dx)
=&\,\sum_{n\geq 1}
  \delta_{\left(k_n,t_n,U_n I(X_{t_n-};k_n),x_n\right)}(dk, dt,dr, dx)\\
&\,+\sum_{n\geq 1} \1_{\{r> I(X_{t_n-};k_n)\}} \CN^1 \big(dk, dt,dr, dx\big).
\end{aligned}
\end{equation}
Both summands in \equ{E5.1} are predictable transformations of point processes of class (QL) (in the sense of \cite[Definition 3.2]{IkedaWatanabe1989}); that is, they possess continuous compensators. Standard arguments yield that they are hence also point processes of class (QL). A standard computation shows that the compensator measures are given by
$$ \ell_S(dk) \,\1_{\{r\leq  I(X_{t-};k)\}} \,\lambda(dt)\,\lambda(dr)\,\nu(dx)
$$
and
$$ \ell_S(dk) \,\1_{\{r\leq  I(X_{t-};k)\}} \,\lambda(dt)\,\lambda(dr)\,\nu(dx),
$$
respectively. Hence $\CN$ is a point process of class (QL) and has the deterministic and absolutely continuous compensator measure
$\ell_S \otimes\lambda\otimes\lambda\otimes\nu$.
By \cite[Theorem 6.2]{IkedaWatanabe1989}, we get that $\CN\,$ is thus a Poisson point process with intensity $\ell_S \otimes\lambda\otimes\lambda\otimes\nu$.
\gdm

\section{Proof of Theorem~\ref{T1.5}}
\label{S6}
\setcounter{equation}{0}
\setcounter{theorem}{0}
Recall that $Y^\gamma$ solves the following system of equations
\begin{equation}
\label{E6.1}
\begin{aligned}
Y^{\gamma}_{i,t}(k)=&\,y_{i,0}(k)+\int_0^t \CA Y^{\gamma}_{i,s}(k)\,ds +\int_0^t \gamma ^{1/2}\sigma( Y^{\gamma}_s(k))\,dW_{i,s}(k),\qquad t\geq 0,\, k\in S,\,i=1,2.
\end{aligned}
\end{equation}

First of all we establish uniform integrability of $Y^\gamma_i$, $i=1,2$.
\begin{lemma}
\label{L6.1}
 For any $T>0$, $p\in(0,2)$ and $i=1,2$, we have
$$
\sup_{\gamma\geq0} \E\left[ \sup_{t\leq T} \big\langle Y^{\gamma}_{i,t}\,,\beta\big\rangle^p\right]<\infty.
$$
\end{lemma}
\paragraph{Proof.}
By simple stochastic calculus, we get
\begin{equation}
\label{E6.2}
\begin{aligned}
e^{-\LSC t}
\big\langle Y^{\gamma}_{i,t}\,,\beta\big\rangle
\;&=\;
\big\langle Y^\gamma_{i,0}\,,\beta\big\rangle+ \int_0^t \left(\big\langle \CA Y^{\gamma}_{i,s}\,, \beta\big\rangle-\LSC\big\langle Y^{\gamma}_{i,s}\,,\beta\big\rangle\right)\,ds\\
&\phantom{=}+\sum_{k\in S}\beta(k)\int_0^t e^{-\LSC s}\,\gamma^{1/2}\, \sigma(Y^{\gamma}_s(k))\,dW_{i,s}(k)\\
&\leq\;
\big\langle Y^\gamma_{i,0}\,,\beta\big\rangle+\sum_{k\in S}\beta(k)\int_0^t e^{-\LSC s}\,\gamma^{1/2}\, \sigma(Y^{\gamma}_s(k))\,dW_{i,s}(k)\\
&\leq\;
\big\langle Y^\gamma_{i,0}\,,\beta\big\rangle+ B_{i,\;\sum_{k\in S}\beta(k)^2\int_0^t e^{-2\LSC s}\,\gamma\, \sigma^2(Y^{\gamma}_s(k))\,ds}
\end{aligned}
\end{equation}
where the second inequality follows by \equ{E1.9} and
 $B_i$, $i=1,2$, are independent Brownian motions. Hence we get that
the pair
\[ \left(e^{-\LSC t}\langle Y^\gamma_{1,t}\,,\beta\rangle, e^{-\LSC t}\langle Y^{\gamma}_{2,t}\,,\beta\rangle\right)\]
is stochastically bounded by the time-changed planar Brownian motion $B$ starting at
 $$B_0=(u,v):=\left( \langle Y^\gamma_{1,0}\,,\beta\rangle, \langle Y^\gamma_{2,0}\,,\beta\rangle\right)$$ and evolving until
 the stopping time
\[ \tau=\inf\{t\geq0:  B_{1,t} B_{2,t}=0\}. \]
For $p\in(1,2)$, by Doob's inequality, we have
\begin{equation}
\label{E6.3}
K_i\equiv \E\left[ \sup_{t\leq \tau}(B_{i,t})^p\right]<\left(\frac{p}{p-1}\right)^p\E \big[( B_{i,\tau})^p\big].
\end{equation}
The $(p/2)$th moment of the exit time of planar Brownian motion from a quadrant is finite if and only if $p<2$ (see, e.g., \cite[Equation (3.8)]{Burkholder1977} with $\alpha=\pi/2$). Hence, using Burkholder's inequality, we get
\begin{equation}
\label{E6.4}
K_i<\infty.
\end{equation}
We can get \equ{E6.4} also by an explicit estimate using the density of the distribution of $\tilde B_\tau$ from \equ{E1.19}:
$$\E\big[(B_{i,\tau})^p\big]\leq |u^2-v^2|^{p/2}\;+\;\frac{2^{p/2}\,(uv)^{p/2}}{\cos(p\pi/4)}.$$

This immediately implies that
\begin{equation}
\label{E6.5}
\E\left[ \sup_{t\leq T}\big\langle Y^{\gamma}_{i,t}\,,\beta\big\rangle^p\right]<
 e^{\LSC T}\, K_i<\infty\mfa i=1,2,
\end{equation}
uniformly in $\gamma\geq0$. \gdm
\begin{lemma}
The family $(Y^\gamma)_{\gamma\geq0}$ is tight in $D_{\LBE}$ equipped with Meyer-Zheng pseudo-path topology.
\end{lemma}
\paragraph{Proof.}
The process $M^{\gamma}_i(k)$ defined by
$$
M^{\gamma}_{i,t}(k):=Y^{\gamma}_{i,t}(k)-Y_{i,0}(k)-\int_0^t \CA Y^{\gamma}_{i,s}(k)\,ds
$$
is a martingale. In order to show tightness of $(Y^\gamma)_{\gamma\geq0}$, it is enough to show
tightness of $(Y^\gamma_i(k))_{\gamma\geq0}$ for all $k\in S$ and $i=1,2.$ By Lemma~\ref{L6.1}, the random variable
$\langle Y^{\gamma}_{i,t}\,,\beta\rangle$ has a $p$-th moment for any $p\in(0,2)$, hence we immediately get tightness
of
\[  \int_0^t \CA Y^{\gamma}_{i,s}(k)\,ds.\]
Note that the conditional variation (see, e.g., \cite[page 358]{bib:mz84}) $V^\gamma_{i,T}(k)$ of $M^{\gamma}_{i,t}(k)$ up to time $T$ equals
$$V^\gamma_{i,T}(k)=\sup_{t\leq T}\E[|M^{\gamma}_{i,t}(k)|].$$
Hence, by Theorem~4 of \cite{bib:mz84},
in order to get the tightness of the martingale $M^{\gamma}_{i,t}(k)$ it is enough to show that
\begin{equation}
\label{E6.6} \sup_{\gamma>0}\;\sup_{t\leq T} \E\left[ \big| M^{\gamma}_{i,t}(k)\big|\right]\mfa T>0.
\end{equation}
However,
$$
|M^{\gamma}_{i,t}(k)| \leq |Y^{\gamma}_{i,t}(k)| +|x_{1}(k)|+\bigg|\int_0^t \CA Y^{\gamma}_{i,s}(k)\,ds\bigg|
$$

and \equ{E6.6} again follows immediately by boundedness of $p$-th moments (for $p<2$) of $\langle Y^{\gamma}_{i,t}\,,\beta\rangle$.
\gdm

\begin{lemma}
Let $X$ be an arbitrary limit point of $(Y^{\gamma})_{\gamma\geq0}$. Then
$X$ solves the martingale problem \equ{MP1}.
\end{lemma}
\paragraph{Proof.}
Let $\gamma_n\to\infty$ be such that $Y^{\gamma_n}$ converges to $X$ as $n\to\infty$.
By It\^o's formula, for $z\in\LFT$ (recall \equ{E1.13}), the process $M^{\gamma,y,z}$ defined by
\begin{equation}
\label{E6.7}
M^{\gamma,y,z}_t=H(Y^{\gamma}_t,z)-H(Y^{\gamma}_0,z)-\int_0^t \lb\CA Y^\gamma,\,z\rb\, H(Y^\gamma_s,z)\,ds
\end{equation}
is a martingale. Since $(Y^{\gamma_n})_{n\in\N}$ converges to $X$, the right hand side of \equ{E6.7} converges to $M^{y,z}$. As the $p$-th moments $\langle Y^\gamma_t\,,\beta\rangle$ (for $p\in(0,2)$) are uniformly bounded (in $\gamma$), also the $p$-th moments of $M^{\gamma,y,z}$ are uniformly bounded. By \cite[Theorem 11]{bib:mz84}, we infer that
$M^{x,z}=\lim_{n\to\infty}M^{\gamma_n,x,z}$ is a martingale. In other words, $X$ is a $[0,\infty)^2$-valued solution to the martingale problem \equ{MP1}.
It remains to show that $ X_t\in E$ for all $t>0$, $k\in S$.
Recall that we derived the tightness of the martingales $M^{\gamma,y,z}_t(k)$. But this implies that the
quadratic variation of $M^{\gamma,y,z}(k)$ is stochastically bounded uniformly in $\gamma$; that
is,
\[ \int_0^t \gamma\, \sigma^2(Y^{\gamma}_s(k))\, ds\]
is uniformly bounded in $\gamma$. Since $\gamma_{n}\rightarrow\infty$, this implies
$$
 \int_0^t  \sigma(Y^{\gamma_n}_{s}(k))\, ds \limn 0.
$$
By Assumption \ref{A1.4}(ii) and \equ{E6.5}, this implies
$$
 \int_0^t \big[ \big(Y^{\gamma_n}_{1,s}(k)\,Y^{\gamma_n}_{2,s}(k)\big)\wedge1 \big]\, ds\limn 0.
$$
On the other hand, we have
$$
 \int_0^t  \big[\big(Y^{\gamma_n}_{1,s}(k)\,Y^{\gamma_n}_{2,s}(k)\big)\wedge1 \big] \, ds\limn \int_0^t  \big[\big(X_{1,s}(k)\,X_{2,s}(k)\big)\wedge 1\big]\, ds,
$$
hence
$$
 \int_0^t  X_{1,s}(k)X_{2,s}(k)\, ds= 0.
$$
Thus
$$
  X_{1,s}(k)X_{2,s}(k)= 0,
$$
for almost every $s$. Since the limiting process $X$ is c{\`a}dl{\`a}g, we have $X_t\in E^S$ for \emph{all} $t\geq0$ almost surely.
\gdm

The above lemma finishes the proof of Theorem~\ref{T1.5}.

\section*{Appendix A: Properties of the jump measure}
\setcounter{theorem}{0}
\setcounter{equation}{0}
\renewcommand{\thetheorem}{A.\arabic{theorem}}
\renewcommand{\theequation}{A.\arabic{equation}}
\label{App:A}
Recall the measure $\nu$ from \equ{E1.20}.
For ease of reference, we collect some basic facts on the moments of $\nu$.
\begin{lemma}
\label{LA.1}
Let $\ve>0$. We have
$$
\nu\big(\{0\}\times(\ve,\infty)\big)\;=\;\frac{2}{\pi}\frac{1}{(1+\ve)^2}\;\leq\; \frac{2}{\pi}\big(1\wedge \ve^{-2}\big),
$$
and
$$
\begin{aligned}
\nu\big(([0,\infty)\setminus(1-\ve,1+\ve))\times \{0\}\big)&=
\cases{
\frac{8}{\pi}\frac{1}{\ve(4-\ve^2)}-\frac2\pi,&\mfalls\ve\leq 1,\\[2mm]
\frac{2}{\pi}\frac{1}{\ve(2+\ve)},&\mfalls \ve\geq 1,
}\\
&\leq\frac{2}{\pi}\big(\ve^{-1}\wedge\ve^{-2}\big).
\end{aligned}
$$
\end{lemma}
\begin{proof}This is simple calculus.\end{proof}
\begin{lemma}
\label{LA.2}
For $x\geq0$, we have
$$\int y_2\,\nu(dy)=1,$$
$$\int_{\{0\}\times(0,x)}y_2\,\nu(dy)\;=\;\frac2\pi\arctan(x)-\frac2\pi\frac{x}{1+x^2},$$
and
$$\int_{\{0\}\times(0,x)}y^2_2\,\nu(dy)\;=\;\frac2\pi\log(1+x^2)-\frac2\pi\frac{x^2}{1+x^2}\;\leq\;\frac4\pi\,\log(x),$$
where the inequality holds if $x\geq2$.
\end{lemma}
\begin{proof}This is simple calculus.\end{proof}
\begin{lemma}
\label{LA.3}
For $\ve>0$, we have
\begin{equation}
\label{EA.1}
\int_{\{y_1\geq 1+\ve\}}(y_1-1)\,\nu(dy)=\frac1\pi\left(1+\log\left(\frac{2+\ve}{\ve}\right)-\frac{\ve}{2+\ve}\right).
\end{equation}
For $\ve\in(0,1)$, we have
\begin{equation}
\label{EA.2}
\int_{\{y_1\leq 1-\ve\}}(y_1-1)\,\nu(dy)=\frac1\pi\left(-1+\log\left(\frac{\ve}{2-\ve}\right)-\frac{\ve}{2-\ve}\right).
\end{equation}
Hence for $\ve\in(0,1)$, there exists an $\ve':=\ve'(\ve)\in[\ve/2,\ve]$ such that
\begin{equation}
\label{EA.3}
\int_{\{y_1\not\in(1-\ve,1+\ve')\}}(y_1-1)\,\nu(dy)=0.
\end{equation}
\end{lemma}
\begin{proof}By elementary calculus, we get \equ{EA.1} and \equ{EA.2}.
Using the explicit expressions, it is easy to check that
$$\int_{\{y_1\not\in(1-\ve,1+\ve)\}}(y_1-1)\,\nu(dy)\leq0
\leq\int_{\{y_1\not\in(1-\ve,1+\ve/2)\}}(y_1-1)\,\nu(dy).
$$
By continuity, \equ{EA.3} holds for some $\ve'\in[\ve/2,\ve]$.
\end{proof}

Note that letting $\ve\to0$ in Lemma~\ref{LA.3} yields that $\int (y_1-1)\nu(dy)=0$ in the sense of a Cauchy principal value.
\begin{lemma}
\label{LA.4}
For $x>0$, we have
$$
\int_{(0,x)\times\{0\}}(y_1-1)^2\,\nu(dy)\;=\;
\frac{4}{\pi}\left(\log(1+x)-\frac{x}{1+x}\right).
$$
Hence, for $\ve\in(0,1)$, we get
$$\int_{(1-\ve,1+\ve)\times\{0\}}(y_1-1)^2\,\nu(dy)
=\frac{4}{\pi}\left(\log\left(\frac{2+\ve}{2-\ve}\right)-\frac{2\ve}{4-\ve^2}\right)
\;\leq\;\frac{2}{\pi}\,\ve$$
\end{lemma}
\begin{proof}This is simple calculus.\end{proof}
\begin{lemma}
\label{LA.5}
For $p\in(1,2)$, we have
\begin{equation}
\label{EA.4}
m_{1,p}:=\int_E|y_1-1|^p\,\nu(dy)
\leq \frac{4}{\pi}\frac{p^2-2p+2}{p(p-1)(2-p)}<\infty.
\end{equation}
and
\begin{equation}
\label{EA.5}
m_{2,p}:=\int_Ey_2^p\,\nu(dy)={\frac {\pi \,p \left(p-1 \right) }{\sin \left( \pi \,p \right) }}<\infty.
\end{equation}
\end{lemma}\begin{proof}
Note that
$$m_{1,p}\leq \frac{4}{\pi}\int_0^\infty \frac{v}{(1+v^2)^2}dv
+\frac{4}{\pi}\int_0^1 u(1-u)^{2-p}\,du
+\frac{4}{\pi}\int_0^\infty \frac{u(1-u)^{2-p}}{(1+u)^2}\,du
$$
and the right hand side equals the right hand side of \equ{EA.4}.
The formula for $m_{2,p}$ can be derived by an explicit calculation using (i).
\end{proof}
\begin{corollary}
\label{CA.6}
Recall $J$ from \equ{E1.23}. For any $L>0$, we have
$$\nu\big(\big\{y:\|J(y,(1,0))\|_\infty\geq L\big\}\big)\leq \frac4\pi \,L^{-2}.$$
\end{corollary}
\begin{proof}
This is a direct consequence of the definition of $J$ (see \equ{E1.23}) and Lemma~\ref{LA.1}.
\end{proof}

\section*{Appendix B: Proof of Lemma~\ref{L1.2}.}
Let $x=(u,v)\in(0,\infty)^2$. Explicit integration of \equ{E1.19} yields
$$Q_{(u,v)}(\{0\}\times[\bar v,\infty))=\frac12+\frac1\pi\arctan\left(\frac{v^2-u^2-\bar v^2}{2uv}\right)$$
and
$$Q_{(u,v)}([\bar u,\infty)\times\{0\})=\frac12+\frac1\pi\arctan\left(\frac{u^2-v^2-\bar u^2}{2uv}\right).$$
This yields
$$Q_{(u,v)}(E)=\frac12+\frac12+\frac1\pi\arctan\left(\frac{v^2-u^2}{2uv}\right)+\frac1\pi\arctan\left(\frac{u^2-v^2}{2uv}\right) =1.$$
Hence $Q$ as defined in \equ{E1.19} is in fact a probability measure. Furthermore, for $u_0>0$ and $\ve\in(0,u_0)$, we have
$$\begin{aligned}
Q_{(u,v)}\big(\{0\}\times(u_0-\ve,u_0+\ve)\big)&=\frac1\pi\arctan\left(\frac{(u_0+\ve)^2-u^2+v^2}{2uv}\right)
-\frac1\pi\arctan\left(\frac{(u_0-\ve)^2-u^2+v^2}{2uv}\right)\\
&\longrightarrow \frac1\pi\arctan(\infty)-\frac1\pi\arctan(-\infty)=1\quad\mbox{as }(u,v)\to (u_0,0).\end{aligned}
$$
Hence $Q_{(u,v)}\to \delta_{(u_0,0)}$ as $(u,v)\to(u_0,0)$. By symmetry, we also have $Q_{(u,v)}\to \delta_{(0,v_0)}$ as $(u,v)\to(0,v_0)$. Finally, explicitly computing second derivatives gives
$$\begin{aligned}
\Big(\frac{\partial^2}{\partial u^2}+&\frac{\partial^2}{\partial v^2}\Big)
\frac{\TS uv\,\bar u}{\textstyle 4u^2v^2+\big(\bar u^2+ v^2-u^2\big)^2}
\\
&=-\frac {4uv\,\bar{u} \left( 14\,{u}^{2}{v}^{2}\,{\bar{u}}^{2}-3\,{u}^{6}-3\,{\bar{u}}^{6}
+3\,{v}^{6}-3\,{u}^{4}{v}^{2}+3\,{u}^{2}{\bar{u}}^{4}+3\,{u}^{4}{\bar{u}}^{2}+3\,{
u}^{2}{v}^{4}-3\,{\bar{u}}^{4}{v}^{2}+3\,{\bar{u}}^{2}{v}^{4} \right) }{ \left( 4u^2v^2+\big(\bar u^2+ v^2-u^2\big)^2 \right) ^{3}}\\
&\phantom{=}+
\frac {4uv\,\bar{u} \left( 14\,{u}^{2}{v}^{2}\,{\bar{u}}^{2}-3\,{u}^{6}-3\,{\bar{u}}^{6}+
3\,{v}^{6}-3\,{u}^{4}{v}^{2}+3\,{u}^{2}{\bar{u}}^{4}+3\,{u}^{4}{\bar{u}}^{2}+3\,{u
}^{2}{v}^{4}-3\,{\bar{u}}^{4}{v}^{2}+3\,{\bar{u}}^{2}{v}^{4} \right) }{ \left( 4u^2v^2+\big(\bar u^2+ v^2-u^2\big)^2 \right) ^{3}}\\
&=0.\end{aligned}
$$
Hence, the function in \equ{E1.19} is indeed harmonic.
\gdm

\section*{Acknowledgement}
We would like to thank an anonymous referee who helped considerably to debug the paper and to improve the exposition.


\begin{thebibliography}{Myt96}
\bibitem[Ald78]{Aldous1978}
D. Aldous.
\newblock  Stopping Times and Tightness.
\newblock {\em Ann. Probab.}, Vol. 6(2), 335--340, 1978.

\bibitem[B97]{bib:b97}
R.~F. Bass.
\newblock {\em Diffusions and {E}lliptic {O}perators.}
\newblock
 Probability and its Applications (New York). Springer-Verlag, New York, 1998.

\bibitem[Bur77]{Burkholder1977}
D.~L. Burkholder.
\newblock Exit times of {B}rownian motion, harmonic majorization, and {H}ardy
  spaces.
\newblock {\em Advances in Math.}, 26(2):182--205, 1977.

\bibitem[CDG04]{bib:cdg04}
J.~T. Cox and D.~A. Dawson and A. Greven.
\newblock  Mutually catalytic super branching random walks: large finite systems and renormalization analysis.
\newblock {\em Mem. Amer. Math. Soc.},  171,  no. 809, 2004.

\bibitem[DP98]{bib:dp98}
D.~A. Dawson and E.~A. Perkins.
\newblock Long time behaviour and co-existence in a mutually catalytic
  branching model.
\newblock {\em The Annals of Probability}, 26(3):1088--1138, 1998.

\bibitem[DM83]{DellacherieMeyer1983.5-8}
C. Dellacherie and P.~A. Meyer.
\newblock {\em Probabilit\'es et potentiel: Chapitres V \`a VIII Th\'eorie des
  martingales}.
\newblock Hermann, Paris, 1983.

\bibitem[EK86]{bib:kur86}
S.~N. Ethier and T.~G. Kurtz.
\newblock {\em Markov {P}rocess: {C}haracterization and {C}onvergence}.
\newblock John Wiley and Sons, New York, 1986.

\bibitem[IW89]{IkedaWatanabe1989}
N. Ikeda and S. Watanabe.
\newblock {\em Stochastic differential equations and diffusion processes},
  volume~24 of {\em North-Holland Mathematical Library}.
\newblock North-Holland Publishing Co., Amsterdam, 2. edition, 1989.

\bibitem[JS87]{bib:jacshir87}
J. Jacod and A.~N. Shiryaev.
\newblock {\em Limit Theorems for Stochastic Processes}.
\newblock Springer-Verlag, New York, 1987.


\bibitem[KM10]{KM1}
A. Klenke and L. Mytnik.
\newblock Infinite rate mutually catalytic branching.
\newblock {\em Ann. Probab.}, 38(4): 1690--1716, 2010.

\bibitem[KM11]{KM3}
A. Klenke and L. Mytnik.
\newblock Infinite rate mutually catalytic branching in infinitely many
  colonies. {T}he longtime behaviour.
\newblock {\em Ann. Probab. (to appear)}, 2011.

\bibitem[KO10]{KO}
A. Klenke and M. Oeler.
\newblock A {T}rotter type approach to infinite rate mutually catalytic
  branching.
\newblock {\em Ann. Probab.}, 38(2): 479--497, 2010.

\bibitem[LM05]{bib:LeGallMytnik}
J.-F.~Le Gall and L.~Mytnik.
\newblock Stochastic integral representation and regularity of the
              density for the exit measure of super-{B}rownian motion,
\newblock\emph{Ann. Probab.} 33(1):194-222, 2005.
\bibitem[Lig85]{Liggett1985}
T.~M. Liggett.
\newblock {\em Interacting particle systems}.
\newblock Springer-Verlag, New York, 1985.

\bibitem[MZ84]{bib:mz84}
P.~A. Meyer and  W.~A. Zheng.
\newblock
 Tightness criteria for laws of semimartingales.
\newblock
{\em Ann. Inst. H. Poincar{\'e} Probab. Statist.},  20(4):353--372, 1984.

\bibitem[Myt96]{bib:myt96}
L. Mytnik.
\newblock Superprocesses in random environments.
\newblock {\em The Annals of Probability}, 24:1953--1978, 1996.

\bibitem[Oel08]{Oeler2008}
M. Oeler.
\newblock  Mutually Catalytic Branching at Infinite Rate.
\newblock {\em PhD thesis, Universit{\"a}t Mainz}, 2008.

\bibitem[Pr04]{bib:pr04}
P.~E. Protter.
\newblock
{\em Stochastic {I}ntegration and {D}ifferential {E}quations.}
\newblock  Springer-Verlag, Berlin, 2004.

\bibitem[SV97]{bib:sv79}
D.~W. Stroock and S.~R. Varadhan.
\newblock
{\em  Multidimensional {D}iffusion {P}rocesses.}
\newblock  Springer-Verlag, Berlin-New York, 1979.

\end{thebibliography}
\end{document}